\documentclass[twoside]{amsart}

\usepackage{amsmath,amssymb,amsfonts,amsthm,latexsym}
\usepackage{amscd,graphicx,color,enumerate}
\usepackage{hyperref}
\usepackage[all]{xy}

\numberwithin{equation}{section}

\newtheorem{theorem}{Theorem}[section]
\newtheorem{lemma}[theorem]{Lemma}
\newtheorem{proposition}[theorem]{Proposition}
\newtheorem{corollary}[theorem]{Corollary}
\newtheorem{conjecture}[theorem]{Conjecture}

\theoremstyle{definition}
\newtheorem{definition}[theorem]{Definition}
\newtheorem{example}[theorem]{Example}
\theoremstyle{remark}
\newtheorem{remark}[theorem]{Remark}

\newcommand{\C}{\mathbb{C}}
\newcommand{\Sp}{\mathbb{S}}
\newcommand{\Q}{\mathbb{Q}}
\newcommand{\R}{\mathbb{R}}

\newcommand{\Z}{\mathbb{Z}}
\newcommand{\N}{\mathbb{N}}

\newcommand{\SU}{\operatorname{SU}}
\newcommand{\SL}{\operatorname{SL}}
\newcommand{\SYMP}{\operatorname{SP}}
\newcommand{\SO}{\operatorname{SO}}
\newcommand{\OO}{\operatorname{O}}
\newcommand{\GL}{\operatorname{GL}}
\newcommand{\id}{\operatorname{id}}

\newcommand{\co}{\colon\thinspace}
\newcommand{\git}{/\!\!/}
\newcommand{\bs}{\boldsymbol}

\newcommand{\Hilb}{\operatorname{Hilb}}

\newcommand{\calO}{\mathcal{O}}
\newcommand{\pr}{\mathrm{pr}}
\newcommand{\sm}{\mathrm{sm}}
\newcommand{\codim}{\operatorname{codim}}
\newcommand{\Spec}{\operatorname{Spec}}

\newcommand{\inv}{^{-1}}
\newcommand{\Cl}{\operatorname{Cl}}
\newcommand{\vol}{\operatorname{vol}}
\newcommand{\NN}{\mathcal N}

\newcommand{\lie}{\operatorname{Lie}}

\begin{document}

\title{Symplectic quotients have symplectic singularities}

\author{Hans-Christian Herbig}
\email{herbighc@gmail.com}
\address{Departamento de Matem\'{a}tica Aplicada, Universidade Federal do Rio de Janeiro,
Av. Athos da Silveira Ramos 149, Centro de Tecnologia - Bloco C, CEP: 21941-909 - Rio de Janeiro, Brazil}

\author{Gerald W. Schwarz}
\email{schwarz@brandeis.edu}
\address{Department of Mathematics, Brandeis University,
Waltham, MA 02454-9110, USA}

\author{Christopher Seaton}
\email{seatonc@rhodes.edu}
\address{Department of Mathematics and Computer Science,
Rhodes College, 2000 N. Parkway, Memphis, TN 38112, USA}


\thanks{C.S. was supported by the E.C.~Ellett Professorship in Mathematics.}

\begin{abstract}
Let $K$ be a compact Lie group with complexification $G$, and let $V$ be a unitary
$K$-module. We consider the real symplectic quotient $M_0$ at level $0$ of the
homogeneous quadratic moment map as well as the complex symplectic quotient, defined
here as
the complexification of $M_0$. We show that if $(V, G)$ is $3$-large, a condition
that holds generically, then the complex symplectic quotient has symplectic singularities
and is graded Gorenstein. This in particular implies that the real symplectic quotient
is graded Gorenstein. In the case that $K$ is a torus or $\operatorname{SU}_2$, we show
that these results hold without the hypothesis that $(V,G)$ is $3$-large.
\end{abstract}

\maketitle

% xxxxxxxxxxxxxxxxxxxxxxxxxxxxxxxxxxxxxxxxxxxxxxxxxxxxxxxxxxxxxxxxxxxxxxxxx
% xxxxxxxxxxxxxxxxxxxxxxxxxxxxxxxxxxxxxxxxxxxxxxxxxxxxxxxxxxxxxxxxxxxxxxxxx
% xxxxxxxxxxxxxxxxxxxxxxxxxxxxxxxxxxxxxxxxxxxxxxxxxxxxxxxxxxxxxxxxxxxxxxxxx

\section{Introduction}
\label{sec:Intro}

Let $K$ be a compact Lie group, let $V$ be a unitary $K$-module and let $G = K_{\C}$
denote the complexification of $K$. We denote the Lie algebras of $K$ and $G$ by
$\mathfrak{k}$ and $\mathfrak{g}$, respectively.
Let $J\co V\to\mathfrak{k}^\ast$ be the associated homogeneous quadratic moment map.
The \emph{real symplectic quotient} at the zero level of $J$ is $M_0 = J^{-1}(0)/K$.
The complexification of $M_0$ is the corresponding \emph{complex symplectic quotient};
see Definitions \ref{def:RSympQuot} and \ref{def:CSympQuot}.
One may equivalently begin with a complex reductive group $G$
and a $G$-module $V$ and then choose a maximal compact subgroup $K$ of $G$ as well as a
Hermitian structure on $V$ with respect to which the $K$-action is unitary.

In this paper, we demonstrate that complex symplectic quotients are
frequently symplectic varieties in the sense of \cite{Beauville}. This in particular implies
that they are
Gorenstein with rational singularities. Moreover, we demonstrate that they are \emph{graded Gorenstein},
sometimes called \emph{strongly Gorenstein}, meaning that they are Gorenstein with no
degree shift in their canonical module. This implies that
the Zariski closure of the real symplectic quotient $M_0$ is also graded Gorenstein.
Note that $M_0$ is in a natural way a semialgebraic set, see Section \ref{subsec:BackSympQuot}.
Then $\R[M_0]$, the ring of real regular functions on $M_0$, is also graded Gorenstein.

In order to state our main result, let us first explain our notation. For a $G$-module $V$,
we let $\mu = J\otimes_{\R}\C\co V\otimes_{\R}\C \simeq V\oplus V^\ast \to \mathfrak{g}^\ast$
denote the complex moment map. Then $\mu(v,\xi)(A) = \xi(A(v))$ for $A \in\mathfrak{g}$ and
$(v,\xi)\in V\oplus V^\ast$. See Definition \ref{def:KLarge} for the definitions of $k$-large, $k$-principal, and stable,
properties of a $G$-module $(V,G)$,
and note that $k$-large implies $(k-1)$-large. If $(V,G)$ is at least $1$-large, then the complex symplectic
quotient is given by the categorical quotient $\mu^{-1}(0)\git G$. Our main result is the following.

\begin{theorem}
\label{thrm:GeneralMain}
Suppose that $(V,G)$ is $3$-large. Then the complex symplectic quotient
$\mu^{-1}(0)\git G$ has symplectic
singularities and is graded Gorenstein. In particular,
$\mu^{-1}(0)\git G$ has rational singularities.
\end{theorem}

Recall that a complex variety $Y$ has rational singularities if it is normal and, for every resolution
of singularities $f\co X\to Y$, $R^i f_\ast \calO_X = 0$ for $i > 0$; see \cite[Section 5.1]{KollarMori}.

For a fixed choice of $K$ and $G$, Theorem \ref{thrm:GeneralMain} implies that the complex symplectic quotient
associated to ``most" $G$-modules has symplectic singularities and is graded Gorenstein.
For example, if $G$ is connected and semisimple, then all but finitely many $V$ such that $V^G=\{0\}$ and
each irreducible component of $V$ is almost faithful (i.e. the kernel of the action is finite) are $3$-large;
see Section \ref{subsec:MostVLarge}. In Section \ref{sec:GeneralMain}, we will also show that the
conclusions of Theorem \ref{thrm:GeneralMain} hold if $(V,G)$ is $2$-large with additional hypotheses.

When $(V,G)$ fails to be $1$-large, there are different ways of realizing $\mu^{-1}(0)\git G$ as a scheme
and hence different definitions of the complex symplectic quotient. It is known that there are examples of $(V,G)$
such that, by some such definitions, the complex symplectic quotient $\mu^{-1}(0)\git G$ fails to have
symplectic singularities, see \cite{Becker,BuloisLehnTerpereau} and Sections
\ref{subsec:BackSympQuot} and \ref{sec:SU2} below. In
this paper, we propose another definition of the complex symplectic quotient for which we
expect Theorem \ref{thrm:GeneralMain} to hold more generally, though the proof will require very different
techniques. Specifically, we define the complex symplectic quotient to be the complexification of the real
quotient, $\Spec \big( \R[M_0]\otimes_{\R}\C\big)$, see Definition \ref{def:CSympQuot}, corresponding to
taking the real radical of $(J)$ before complexifying. Then for the few $(V,G)$ that fail to satisfy the
hypotheses of Theorem \ref{thrm:GeneralMain}, we make the following.
\begin{conjecture}
\label{conj:Main}
For all $(V,G)$, the (modified) complex symplectic quotient of Definition \ref{def:CSympQuot}
has symplectic singularities and is graded Gorenstein.
\end{conjecture}
Other authors define the complex symplectic quotient to be either $\Spec \big(\C[V\oplus V^\ast]^G/(\mu)^G\big)$ or
$\Spec \big(\C[V\oplus V^\ast]^G/\big(\sqrt{\mu}\big)^G\big)$, which can fail to have symplectic
singularities or even be reduced or Cohen-Macaulay. See Section \ref{subsec:BackSympQuot} for a discussion
of these definitions as well as examples where they do not coincide.
Using our definition, we are not aware of any counterexamples to Conjecture \ref{conj:Main}.
It is important to note that the three definitions yield the same scheme in the case of
$k$-large representations for $k\geq 1$; see Lemma \ref{lem:R=CSympQuot}. Hence, as $(V,G)$ is assumed $3$-large
in Theorem \ref{thrm:GeneralMain}, the notation $\mu^{-1}(0)\git G$ is unambiguous, and this result holds
for any of the three definitions.

One piece of evidence in
the direction of Conjecture \ref{conj:Main} is related to the question of which (real) symplectic
quotients are (graded regularly) symplectomorphic to a symplectic orbifold, the quotient of a
real symplectic vector
space by a finite group of symplectomorphisms. The authors have demonstrated in
\cite[Theorem 1.3]{HerbigSchwarzSeaton} that such isomorphisms are rare: one can only exist if
$(V,G)$ fails to be $2$-large (or even $2$-principal and stable if $G$ is connected,
\cite[Theorem 1.1]{HerbigSchwarzSeaton}). This has similar implications for the complex
symplectic quotients; see Section \ref{sec:Background}.
A symplectic orbifold has symplectic singularities by \cite[Proposition 2.4]{Beauville} and is graded Gorenstein by
the work of Watanabe \cite{WatanabeGorenstein}; see \cite[Theorem 7.1]{StanleyInvarFinGrp} and
\cite[Section 6.1]{HerbigHerdenSeaton}. Hence, only symplectic quotients associated to ``small" $G$-modules
inherit the properties of being symplectic varieties and graded Gorenstein via an isomorphism with an orbifold,
while Theorem \ref{thrm:GeneralMain} demonstrates these properties for ``large" $G$-modules. A major part of
the motivation of this investigation is the observation that symplectic quotients are rarely isomorphic
to orbifolds yet appear to share many of the properties of orbifolds.
While our emphasis is on ``large" $G$-modules, the recent work of
Bulois, Lehn, Lehn, and Terpereau \cite{BuloisLehnTerpereau} focuses on defining the complex
symplectic quotients for ``small" $G$-modules, specifically ones which are polar. For small cases, the three definitions of the complex symplectic quotient can yield different results,
and the quotients considered in \cite{BuloisLehnTerpereau} are quite often orbifolds.

Other evidence for
Conjecture \ref{conj:Main} is as follows.
In \cite[Section 8.3]{HerbigSeaton}, the Hilbert series of the ring of regular functions on the (real)
symplectic quotients corresponding to several non-$1$-large representations of
$\OO_n(\C)$ and $\SU_2$ were
computed and seen to be so-called \emph{symplectic} Hilbert series; this condition was later
identified in \cite[Corollary 1.8]{HerbigHerdenSeaton} to be equivalent to the ring being graded Gorenstein.
In \cite{CapeHerbigSeaton}, symplectic quotients corresponding to sums of the standard representations of
$\OO_n(\C)$ and $\SO_n(\C)$ were studied.
For small cases, the graded
Gorenstein property was verified by explicit computation of the Hilbert series.  It was demonstrated that, among these representations,
$\mu^{-1}(0)$ already has rational singularities if the representation
$(V,\SO_n(\C))$ is $2$-large. We will see other
examples of $(V,G)$ such that $\mu^{-1}(0)$ has rational singularities; see Proposition
\ref{prop:TorusRationalShell} and Section \ref{sec:SU2}.

The complex symplectic quotients corresponding to classical representations of
$\GL_n(\C)$, $\OO_n(\C)$, and
$\operatorname{Sp}_n(\C)$ have been shown to have symplectic singularities in
\cite[Appendix A.2]{TerpereauThesis},
see also \cite{BuloisLehnTerpereau,TerpereauHilbSch}. Becker \cite{Becker} has demonstrated that
the complex symplectic quotients associated to certain representations of $\SL_2(\C)$
are symplectic varieties as well.
Note that these references consider also the construction of
\emph{symplectic resolutions} of these symplectic quotients. The existence of symplectic resolutions appears to
be rare among symplectic varieties; see \cite[Theorem 3.2]{VerbitskyHoloSymp}.
Bellamy and Schedler \cite{BellamySchedlerQuiver} have demonstrated that
Nakajima quiver varieties are symplectic varieties and identified which admit symplectic resolutions.
More recently and using the results contained in this paper, they have shown in \cite{BellamySchedler} that
if $(V,G)$ is $3$-large,
$G/[G,G]$ is finite, and $G$ acts nontrivially on $V$, then the complex symplectic quotient
does not admit a symplectic resolution. They have further demonstrated that for $(V,G)$ $3$-large,
the complex symplectic quotient has terminal singularities, is $\Q$-factorial if and only if
$G/[G,G]$ is finite, and is locally factorial when $G/[G,G]$ is trivial.
As explained above, the definition of
complex symplectic quotient is different than ours in the above references, though the definitions coincide
in $1$-large cases.

To further illustrate Conjecture \ref{conj:Main}, we consider non-$3$-large $G$-modules for certain $G$.
We first consider the case of a torus, yielding the following.

\begin{theorem}
\label{thrm:TorusMain}
Assume the identity component $K^0$ of $K$ is a torus. Then the corresponding (modified)
complex symplectic quotient of Definition \ref{def:CSympQuot} has
symplectic singularities and is graded Gorenstein. In particular,
the complex symplectic quotient has rational singularities.
\end{theorem}

Note that in the case that $(V,K)$ is unimodular,
meaning that all maximal minors of
the weight matrix of the action are $\pm 1$ or $0$,
Theorem \ref{thrm:TorusMain}
was established in \cite[Proposition 4.11]{BellamyKuwabara}.
See also \cite[Section 3]{Bulois}, where it is demonstrated that when $G$ is a torus,
the alternate definition given in Equation \eqref{eq:defCSympQuotComplexGeometric}
yields complex symplectic quotients
that are in general reduced and normal.

Similarly, for $G = \SL_2(\C)$,
we will indicate that Theorem \ref{thrm:GeneralMain} applies as well when $(V,G)$ is
$2$-large. The few $G$-modules that fail to be $2$-large have been determined
in \cite[Theorem 11.9]{GWSlifting}, and many of
the corresponding symplectic quotients
were found to be isomorphic to linear symplectic
orbifolds in \cite[Theorem 1.6]{HerbigSchwarzSeaton}. Considering the remaining few, we obtain the following.

\begin{theorem}
\label{thrm:SU2Main}
Let $K = \SU_2$ and let $(V,K)$ be unitary. Then the
the corresponding (modified) complex symplectic quotient of Definition \ref{def:CSympQuot}
has symplectic singularities and is graded Gorenstein. In particular,
the complex symplectic quotient has rational singularities.
\end{theorem}

The outline of this paper is as follows.
In Section \ref{sec:Background}, we recall the relevant background information we need.
Section \ref{subsec:BackKLarge} gives definitions of the hypotheses for $G$-modules that
will play a role in this work, Section \ref{subsec:BackSympQuot} explains our definition
of the real and complex symplectic quotients, and Section \ref{subsec:BackSympSing} recalls
the background and results on symplectic and Gorenstein singularities. In Section \ref{sec:AuxResults}
we establish several auxiliary results we will need about the structure of
$\mu^{-1}(0)$ and
the properties of $G$-modules. In particular,
in Section \ref{subsec:SympSlice} we give a symplectic slice theorem for complex
symplectic quotients
similar to that of \cite[Section 6]{BuloisLehnTerpereau} and use it to demonstrate that for
many $G$-modules, the complex symplectic quotient inherits a symplectic structure on
its smooth locus; see Corollary \ref{cor:SympFormSmooth}. We then prove Theorem \ref{thrm:GeneralMain}
in Section \ref{sec:GeneralMain}. In Section \ref{sec:Torus}, we consider the case of $G^0$ a torus
and prove Theorem \ref{thrm:TorusMain}, while in Section \ref{sec:SU2}, we consider the case of
$K = \SU_2$ and prove Theorem \ref{thrm:SU2Main}.

% xxxxxxxxxxxxxxxxxxxxxxxxxxxxxxxxxxxxxxxxxxxxxxxxxxxxxxxxxxxxxxxxxxxxxxxxx
% xxxxxxxxxxxxxxxxxxxxxxxxxxxxxxxxxxxxxxxxxxxxxxxxxxxxxxxxxxxxxxxxxxxxxxxxx
% xxxxxxxxxxxxxxxxxxxxxxxxxxxxxxxxxxxxxxxxxxxxxxxxxxxxxxxxxxxxxxxxxxxxxxxxx

\section*{Acknowledgements}

The authors would like to thank Arnaud Beauville,
William Fulton, and Hanspeter Kraft for helpful responses to our questions
as well as Micha\"{e}l Bulois for helpful comments and references and pointing out an error in
an earlier draft of this paper.
We would also like to thank the referees for a careful reading of the manuscript and for
many suggestions that improved this paper.

% xxxxxxxxxxxxxxxxxxxxxxxxxxxxxxxxxxxxxxxxxxxxxxxxxxxxxxxxxxxxxxxxxxxxxxxxx
% xxxxxxxxxxxxxxxxxxxxxxxxxxxxxxxxxxxxxxxxxxxxxxxxxxxxxxxxxxxxxxxxxxxxxxxxx
% xxxxxxxxxxxxxxxxxxxxxxxxxxxxxxxxxxxxxxxxxxxxxxxxxxxxxxxxxxxxxxxxxxxxxxxxx

\section{Background and definitions}
\label{sec:Background}

% xxxxxxxxxxxxxxxxxxxxxxxxxxxxxxxxxxxxxxxxxxxxxxxxxxxxxxxxxxxxxxxxxxxxxxxxx

\subsection{Properties of $(V,G)$}
\label{subsec:BackKLarge}

Throughout this paper, $K$ will denote a compact Lie group and $G = K_{\C}$
its complexification, i.e., $G$ is a complex reductive group and $K$
is a maximal compact subgroup of $G$. Let $V$ be a $G$-module. The
categorical quotient
$V\git G = \Spec\big( \C[V]^G\big)$ is an affine variety that parameterizes the closed
orbits in $G$. If $Gv$ is a closed orbit, then the isotropy group $G_v$ is reductive by
Matsushima's theorem \cite{Matsushima,LunaClosed}, and the \emph{isotropy type\/} of the  orbit $Gv$ is the conjugacy class of $G_v$. The variety $V\git G$ is stratified by isotropy type \cite{LunaSlice}, and there is a unique open
stratum $(V\git G)_{\pr}$ called the \emph{principal stratum}; the isotropy groups of points
in the corresponding closed orbits are called \emph{principal isotropy groups}. We let
$V_{\pr} = \pi^{-1}((V\git G)_{\pr})$ where $\pi\co V\to V\git G$ denotes the orbit map, the map dual to the inclusion
$\C[V]^G\subset\C[V]$.
We say that a subset of $V$ is \emph{$G$-saturated\/} if it is a union of fibers of $\pi$.

\begin{definition}
\label{def:KLarge}
Let $V$ be a $G$-module. We say that $(V,G)$ has \emph{FPIG} if the principal isotropy groups are finite
and \emph{TPIG} if they are trivial. We say that  $(V,G)$ is \emph{stable} if there is an open set of closed
orbits; equivalently, if $V_{\pr}$ consists of closed orbits. For $k\geq 1$, we say that $(V,G)$ is
\emph{$k$-principal} if $\codim_\C V\smallsetminus V_{\pr} \geq k$.
Letting $V_{(r)}$ denote the locally closed set of points with isotropy group of complex dimension $r$,
$(V,G)$ is \emph{$k$-modular} if $\codim_\C V_{(r)}\geq r+k$ for $1\leq r\leq \dim_\C G$.
Note in particular that if $(V,G)$ is $k$-modular for $k\geq 0$, then $V_{(0)}\neq\emptyset$.
Finally, $(V,G)$ is \emph{$k$-large} if it is $k$-principal, $k$-modular, and has FPIG.
\end{definition}

See \cite{GWSlifting} for more background on these concepts, and note that references
sometimes differ about whether FPIG is required as part of the definition of
$k$-modular or $k$-principal.

Applying the definition of $k$-modular to $V^G$, we have the following.

\begin{lemma}
\label{lem:dimVminusdimG}
Let $V$ be a $k$-modular $G$-module where $G$ is not finite and  $k\geq 0$. Then $\dim_\C G+k\leq\dim_\C V-\dim_\C V^G$.
\end{lemma}
\begin{proof}
As $V^G \subset V_{(\dim_\C G)}$, we have $\dim_\C V^G \leq \dim_\C V_{(\dim_\C G)} \leq \dim_\C V - \dim_\C G - k$.
\end{proof}

% xxxxxxxxxxxxxxxxxxxxxxxxxxxxxxxxxxxxxxxxxxxxxxxxxxxxxxxxxxxxxxxxxxxxxxxxx

\subsection{Real and complex symplectic quotients}
\label{subsec:BackSympQuot}

For a unitary $K$-module $V$ with $\dim_\C V = n$,
let $z_1,\ldots,z_n$ be a choice of unitary coordinates for $V$.
Considering $V$ with its underlying structure as a $2n$-dimensional
real vector space, we equip $V$ with the symplectic form
\[
    \omega = \omega_V
    = \frac{\sqrt{-1}}{2} \sum_{j=1}^n dz_j \wedge d\overline{z}_j,
\]
and then the corresponding Poisson bracket on $\mathcal{C}^\infty(V)$
satisfies
\begin{equation}
\label{eq:PoissonR}
    \{ z_i, \overline{z}_j\} = \frac{2}{\sqrt{-1}} \delta_{i,j},
    \quad\quad
    \{ z_i, z_j\} = \{ \overline{z}_i, \overline{z}_j\} = 0.
\end{equation}
The action of $K$ on $V$ is Hamiltonian with moment map
$J\co V\to\mathfrak{k}^\ast$ defined by
\[
    \langle J(v), A\rangle = \frac{\sqrt{-1}}{2}\langle v, A v \rangle,
    \quad\quad v\in V, A\in\mathfrak{k}.
\]
Choosing a basis $A_1,\ldots,A_m$ for $\mathfrak{k}$, the
components $J_i:= \langle J(\cdot), A_i\rangle$ are homogeneous quadratic
elements of the algebra $\R[V]$ of real regular functions on $V$. We let
$(J)$ denote the ideal generated by the $J_i$ in $\R[V]$ and let $M$
denote the  zero set of $J$. Let $\R[V]^K$ denote the $K$-invariant polynomials on $V$,
a finitely generated graded algebra
\cite[Chapter 8, Section 14]{Weyl}, and let $(J)^K := (J)\cap \R[V]^K$ denote
the invariant part of $(J)$.

\begin{definition}
\label{def:RSympQuot}
Let $K$ be a compact Lie group and $V$ a unitary $K$-module. Let $J\co V\to\mathfrak{k}^\ast$
denote the associated homogeneous quadratic moment map. The \emph{(real) shell   associated to
$(V,K)$} is the set $M$ with ideal $\sqrt[\R]{(J)}$, the real radical of $J$, which is the same thing as the ideal of $\R[V]$ vanishing on $M$. The \emph{real symplectic quotient associated to
$(V,K)$} is the quotient $M_0 := M/K$. Thus the \emph{Poisson algebra of real regular functions on the real symplectic
quotient} is defined by
\[
    \R[M_0] := \R[V]^K/\big(\sqrt[\R]{(J)}\big)^K
\]
where $\big(\sqrt[\R]{(J)}\big)^K := \sqrt[\R]{(J)}\cap\R[V]^K$.
\end{definition}

Note that as $J$ is equivariant, the Poisson bracket on $\R[V]$ induces a well-defined Poisson
bracket on $\R[M_0]$. The real symplectic quotient is a symplectic stratified space \cite{SjamaarLerman}
and is a semialgebraic subset of the semialgebraic set $V/K$
\cite{ProcesiSchwarz}. See also
\cite{ArmsGotayJennings,FarHerSea,HerbigSchwarz,HerbigSeaton} for background
on real symplectic quotients and their algebra of real regular functions.

\begin{remark}
\label{rem:RSympQuotRadical}
It is frequently the case that $J$ generates a \emph{real  ideal\/} of $\R[V]$, i.e.,  $\sqrt[\R]{(J)} = (J)$.
Recalling that $G = K_{\C}$, this is in particular true if $(V,G)$ is $1$-large by \cite[Corollary 4.3]{HerbigSchwarz}.
Even if this condition fails, it can happen that   the invariant part of the real radical of $(J)$ is the
invariant part of $(J)$, i.e. $\big(\sqrt[\R]{(J)}\big)^K = (J)^K$; see Example \ref{ex:HarmOscRC} below.
In these cases, we of course have
\[
    \R[M_0] = \R[V]^K/(J)^K.
\]
\end{remark}

The action of $K$ on $V$ extends to an action of $G$ on $V$.
We complexify the real vector space $V$ to form $V_{\C}:= V\otimes_{\R}\C$,
which is isomorphic   as a $K$-module  to $V\oplus V^\ast$, and we take the $G$-action to be the obvious one on $V\oplus V^\ast$. Letting $(w_1,\ldots,w_n)$ denote coordinates dual to $(z_1,\ldots,z_n)$,
we equip $V\oplus V^\ast$ with the complex symplectic form
\[
    \omega = \omega_{V_{\C}}= \sum\limits_{j=1}^n dz_j \wedge dw_j.
\]
The action of $G$ on $V\oplus V^\ast$ is Hamiltonian with moment map
$\mu = J\otimes_{\R}\C\co V\oplus V^\ast\to \mathfrak{g}^\ast$,
the complexification of the real moment map,
which is given by $\mu(v,\xi)(A) = \xi(A(v))$ for $A \in\mathfrak{g}$ and
$(v,\xi)\in V\oplus V^\ast$, see \cite[Section 4]{HerbigSchwarz}.
The \emph{(complex) shell $N$ associated to $(V,G)$} is the
subscheme of $V\oplus V^\ast$ associated to
the ideal $(\mu)$ generated by the components of $\mu$. With respect to a  basis for
$\mathfrak{g}$, we let $\mu_i$ denote the corresponding coordinates of $\mu$,
and then the $\mu_i$ are homogeneous quadratic elements of the regular functions
$\C[V\oplus V^\ast]$. Note that if the chosen basis consists of real vectors
so that it is also a basis for $\mathfrak{k}$, then $\mu_i = J_i\otimes_{\R}\C$
for each $i$.
We define the complex symplectic quotient to be the complexification of
$M_0$ as follows.

\begin{definition}
\label{def:CSympQuot}
Let $K$ be a compact Lie group, $V$ a unitary $K$-module and $G=K_\C$. Let
$\mu\co V\oplus V^\ast\to\mathfrak{g}^\ast$ denote the moment map associated to the Hamiltonian action
of $G$ on $V\oplus V^\ast$ with respect to its standard symplectic structure, i.e., $\mu = J\otimes_{\R}\C$.
The \emph{complex symplectic quotient associated to $(V,G)$} is
the complex algebraic variety
\begin{equation}
\label{eq:defCSympQuot}
    \Spec \big( \R[M_0]\otimes_{\R}\C\big).
\end{equation}
The Poisson bracket on $\R[M_0]\otimes_{\R}\C$ is that inherited from the bracket on
$\C[V\oplus V^\ast]$
so that
\begin{equation}
\label{eq:PoissonC}
    \{ z_i, w_j\} = \frac{2}{\sqrt{-1}} \delta_{i,j},
    \quad\quad
    \{ z_i, z_j\} = \{ w_i, w_j\} = 0
\end{equation}
where the $z_i$ are coordinates for $V$ and $w_i$ the dual coordinates for $V^\ast$.
\end{definition}

This definition of complex symplectic quotient is not standard. Some authors consider the
alternative quotient
\begin{equation}
\label{eq:defCSympQuotAlgebraic}
    \Spec \big(\C[V\oplus V^\ast]^G/(\mu)^G\big),
\end{equation}
while others compute the radical in the complex sense and consider
\begin{equation}
\label{eq:defCSympQuotComplexGeometric}
      N\git G = \Spec \Big(\C[V\oplus V^\ast]^G/\big(\sqrt{(\mu)}\big)^G \Big).
\end{equation}
Definition \ref{def:CSympQuot}, on the other hand, corresponds to computing the real
symplectic reduction and then complexifying. More precisely, using
\cite[Proposition 5.8(1)]{GWSliftingHomotopies}, the variety of Equation \eqref{eq:defCSympQuot} is
\[
    N^\prime\git G =
    \Spec \Big(\C[V\oplus V^\ast]^G/\big(\sqrt[\R]{(J)}\otimes_{\R}\C\big)^G\Big),
\]
where $N^\prime$ is the variety of $\sqrt[\R]{(J)}\otimes_{\R}\C$ in $V\oplus V^\ast$,
the complexification of the real shell.
Note that in general, $\sqrt{(\mu)} \subset \sqrt[\R]{(J)}\otimes_{\R}\C$, and this inclusion may
be strict; see Example \ref{ex:SL2-2R1}. It follows that $\sqrt{(\mu)}^G \subset \sqrt[\R]{(J)}^G\otimes_{\R}\C$.
It can happen that $\sqrt{(\mu)} \neq \sqrt[\R]{(J)}\otimes_{\R}\C$ while
$\sqrt{(\mu)}^G = \sqrt[\R]{(J)}^G\otimes_{\R}\C$; see Examples \ref{ex:SL2-R1} and \ref{ex:HarmOscRC}
below.

Our motivation for using Definition \ref{def:CSympQuot} is two-fold. First, we are primarily interested
in using the complex symplectic quotient to study properties of the real quotient $M_0$
and its ring $\R[M_0]$ of real regular functions. Secondly, there are cases where
the quotients defined by Equations \eqref{eq:defCSympQuotAlgebraic}
and \eqref{eq:defCSympQuotComplexGeometric} can be pathological in ways that do not
appear to occur for complex symplectic quotients as defined in Definition \ref{def:CSympQuot}.
For example, Theorem \ref{thrm:SU2Main} fails using either quotient given in Equations
\eqref{eq:defCSympQuotAlgebraic} and \eqref{eq:defCSympQuotComplexGeometric}. We illustrate
this fact and the differences between these three definitions with the following.

\begin{example}
\label{ex:SL2-R1}
Let $G = \SL_2(\C)$, and let $V = R_1 = \C^2$ denote the defining representation of $G$ on binary forms of degree $1$.
In standard coordinates $(z_1, z_2, w_1, w_2)$ for $V\oplus V^\ast$, the components of the moment map are given by
$\mu_1 = z_1 w_2$, $\mu_2 = z_2 w_1$, and $\mu_3 = z_1 w_1 - z_2 w_2$.
In this case, $(\mu)$ is not radical and hence $(J)$ is not real.
However, the invariant part of the radical of $(\mu)$ is the
augmentation ideal in the invariant ring
$\C[V\oplus V^\ast]^{\SL_2(\C)}$,
and the invariant part of the real radical of $(J)$ is the
augmentation ideal in $\R[V]^K$. Hence,
$\R[M_0] = \R[V]^K/\big(\sqrt[\R]{(J)}\big)^K = \R$ and $\R[M_0]\otimes_{\R}\C = \C$ so that
the real and complex symplectic quotients are both points.
See \cite[Example 7.12]{ArmsGotayJennings} for a careful discussion.

In \cite[Section 3]{Becker}, Becker considers the complex symplectic quotient associated to
\linebreak
$(R_1, \SL_2(\C))$ using the definition of Equation \eqref{eq:defCSympQuotAlgebraic}. It is demonstrated that
this quotient is not a reduced scheme and hence \emph{not} a symplectic variety. In particular,
\[
    [\C[V\oplus V^\ast]/(\mu)]^{\SL_2(\C)}
    =   \C[V\oplus V^\ast]^{\SL_2(\C)}/(\mu)^{\SL_2(\C)}
    =   \C[z]/(z^2).
\]
Becker corrects for this by equipping the quotient $\Spec\big(\C[z]/(z^2)\big)$ with its reduced
structure, yielding a point. Using the quotient in Equation \eqref{eq:defCSympQuotComplexGeometric},
one obtains
\[
    [\C[V\oplus V^\ast]/\sqrt{(\mu)}]^{\SL_2(\C)}
    =   \C[V\oplus V^\ast]^{\SL_2(\C)}/\big(\sqrt{(\mu)}\big)^{\SL_2(\C)}
    =   \C,
\]
which also yields a point and coincides with our Definition \ref{def:CSympQuot} of the complex
symplectic quotient in this case.
\end{example}

\begin{example}
\label{ex:SL2-2R1}
Let $G = \SL_2(\C)$ and
$V = R_1^{\oplus 2}$ using the notation of Example \ref{ex:SL2-R1}.
The ideal $(\mu)$ is radical, but the ideal $(J)$ is not real;
see \cite[Example 7.13]{ArmsGotayJennings}. In that reference, the real radical
$\sqrt[\R]{(J)}$ of $(J)$ is computed, and it is shown
(\cite[Examples 5.11(a) and 7.6]{ArmsGotayJennings}) that the corresponding real
symplectic quotient $M_0$ is isomorphic to the orbifold $\C/(\pm 1)$,
and hence the complex symplectic quotient is given by $\C[p_1,p_2,p_3]/(p_1 p_2 - p_3^2)$.
This can also be demonstrated using the fact that, for
$V = R_1^{\oplus 2}$, the quotient
$V\git\SL_2(\C)$ is $1$-dimensional over $\C$ so that $V$ is polar; see
\cite[Section 5]{HerbigSchwarzSeaton}.

In this case, however, because $(\mu)$ is radical, the quotient defined in
Equation \eqref{eq:defCSympQuotAlgebraic} is isomorphic to the quotient defined in
Equation \eqref{eq:defCSympQuotComplexGeometric} yet does not coincide
with our Definition \ref{def:CSympQuot} of the complex symplectic quotient. Moreover,
these alternative quotients satisfy none of the consequences of
Theorem \ref{thrm:GeneralMain}. In particular, using the explicit description of the moment
map and invariants given in \cite[Example 7.13]{ArmsGotayJennings}, one computes that
$\C[V\oplus V^\ast]^{\SL_2(\C)}/(\mu)^{\SL_2(\C)}$ is generated by six quadratic polynomials
$\sigma_i$, $i=1,\ldots,6$ subject to the
eleven relations
\begin{align*}
    &\sigma_3\sigma_6,\quad
    \sigma_2\sigma_6,\quad
    \sigma_1\sigma_6,\quad
    \sigma_3\sigma_5,\quad
    \sigma_2\sigma_5,\quad
    \sigma_1\sigma_5,
    \\&
    \sigma_4^2+\sigma_5^2+\sigma_6^2,\quad
    \sigma_3\sigma_4,\quad
    \sigma_2\sigma_4,\quad
    \sigma_1\sigma_4,\quad
    \sigma_1^2-\sigma_2^2-\sigma_3^2.
\end{align*}
This ring has depth $1$ and Krull dimension $2$, so the alternate quotient
fails even to be Cohen-Macaulay.
See \cite[Example 8.7(2)]{BuloisLehnTerpereau}.
\end{example}

We would like to stress, however, that these three notions of complex symplectic quotients do coincide
for almost all representations. Specifically, we have the following.

\begin{lemma}
\label{lem:R=CSympQuot}
Let $(V,K)$ be unitary and let $G = K_{\C}$. If $J$ generates a real   ideal of $\R[V]$, then
\[
    \R[M_0]\otimes_{\R}\C
    \simeq  \C[V\oplus V^\ast]^G/\big(\sqrt{\mu}\big)^G
    =       \C[V\oplus V^\ast]^G/(\mu)^G.
\]
 In particular, this holds if $(V,G)$ is $1$-large.
\end{lemma}
\begin{proof}
By \cite[Proposition 5.8(1)]{GWSliftingHomotopies}, we have
\[
    \C[V\oplus V^\ast]^G
    =       \C[V\oplus V^\ast]^K
    \simeq  \R[V]^K \otimes_{\R} \C.
\]
If $J$ generates a real ideal in $\R[V]$, then $\mu = J\otimes_{\R}\C$ generates a radical
ideal in
$\C[V\oplus V^\ast]$ by \cite[Theorem 6.5]{ArmsGotayJennings}. Since $(\mu)$ is $G$-stable, $(\mu)^G:= (\mu)\cap\C[V\oplus V^\ast]^G$ is radical in $\C[V\oplus V^\ast]^G$.
Hence\[
    \R[M_0]\otimes_{\R}\C
    \simeq  (\R[V]^K\otimes_{\R}\C)/\big((J)^K\otimes_{\R}\C\big)
    \simeq  \C[V\oplus V^\ast]^G/(\mu)^G.
\]
That the hypotheses hold when $(V,G)$ is $1$-large is a consequence of
\cite[Corollary 4.3]{HerbigSchwarz} as noted in Remark \ref{rem:RSympQuotRadical}.
\end{proof}

Even in the rare cases that $(V,G)$ fails to be $1$-large, it can happen that Equations
\eqref{eq:defCSympQuotAlgebraic} and \eqref{eq:defCSympQuotComplexGeometric} coincide
with our definition of the complex symplectic quotient. For instance, there are non-$1$-large
representations for which $J$ generates a real ideal, including the standard action of
$\OO_n(\C)$ or $\SO_n(\C)$ on $V = \C^n$ with $n \geq 3$; see \cite[Example 7.10]{ArmsGotayJennings}.
However, even if the ideal generated by $J$ fails to be real, it can happen that
$(J)^K = \big(\sqrt[\R]{(J)}\big)^K$, as illustrated by the following.

\begin{example}
\label{ex:HarmOscRC}
Let $V = \C$ with $G = \C^\times$-action given by multiplication. Then the only closed orbit
is $\{0\}$ so that $(V,G)$ fails to be stable
so cannot
be $1$-large. The $G$-action on $V\oplus V^\ast = \C^2$ has weight vector $(1,-1)$. In
real coordinates $(z,\overline{z})$ for $V$, the $K = \Sp^1$-moment map $J$ is, up to a
scalar,   $z\overline{z} = |z|^2$. The corresponding ideal $(J) = (z\overline{z})$ is clearly
not real, as its zero set is the origin and its real radical contains $z$ and $\overline{z}$.
On the other hand, using coordinates $(z,w)$ for $V\oplus V^\ast$, the complex moment map
$\mu = zw$ generates the radical ideal $(zw)$.

However, the invariants of this action are generated by $z\overline{z}$ in the real algebra
$\R[V]$ and $zw$ in the complex algebra
$\C[V\oplus V^\ast]$. Therefore,
\[
    (J)^K   =   \big(\sqrt[\R]{(J)}\big)^K  =   J\cdot\R[J],
\]
and
\[
    (\mu)^G   =   \big(\sqrt{(\mu)}\big)^G  =   \mu\cdot\C[\mu].
\]
Hence, similar to the case of Example \ref{ex:SL2-R1}, the real and complex symplectic quotients are both a point.
\end{example}

As another large family of examples, the closure $\overline{\mathcal{O}}$ of a
nilpotent coadjoint orbit $\mathcal{O}$ by one of the classical Lie groups
$\GL_n(\C)$, $\OO_n(\C)$, or $\SYMP_n(\C)$ has been shown by Kraft and Procesi
\cite{KraftProcesiConjClassNorm,KraftProcesiGeomConjClass}
to be isomorphic to what turns out to be a complex symplectic quotient using the
definition in Equation \eqref{eq:defCSympQuotComplexGeometric}; see
\cite[Section 4]{BrylinskiKPI}. Though $\overline{\mathcal{O}}$ may fail to be normal,
its normalization has symplectic singularities by \cite{PanyushevRatGorenNilOrb},
see also \cite{FuNilOrb}. Hence, the corresponding complex symplectic quotients
via Equation \eqref{eq:defCSympQuotAlgebraic} have
symplectic singularities if and only if they are normal. The non-normal cases
are associated to representations that must fail to be $2$-modular,
see Proposition \ref{prop:NormalityShell}({\it iii}); in some cases, the orbit closures are
reducible and hence correspond to representations that cannot be $1$-modular
by Proposition \ref{prop:NormalityShell}({\it ii}), see
\cite[14.2]{KraftProcesiGeomConjClass}. Therefore, the complex symplectic quotients given by
Equation \eqref{eq:defCSympQuotAlgebraic} and Definition \ref{def:CSympQuot} may not
coincide; whether they do has yet to be determined and will be considered elsewhere.

\begin{remark}
\label{rem:IsotropyRvsV}
If $(V,G)$ is $1$-large, then $M$ is Zariski dense in $N$ by \cite[Corollary 4.2]{HerbigSchwarz}.
If $(V,G)$ fails to be $1$-large, then it is possible that $\dim_{\R} M < \dim_{\C} N$
as illustrated in Example \ref{ex:HarmOscRC}.
If $H$ is the isotropy group of a $K$-orbit in $M$, then
$H_\C$ is the isotropy group of the corresponding closed $G$-orbit in $N$.
However, the conjugacy classes of such isotropy groups do not necessarily exhaust
those in $N$, as the following example shows.
\end{remark}

\begin{example}
\label{ex:IsotropyRvsV}
Let $K=\SU_n$, $G=\SL_n(\C)$ and
$V=(\C^n)^{\oplus 2n}$ with the diagonal action.
By \cite[Theorem 11.15]{GWSlifting}, $V$ is $2$-large.
The isotropy groups of $K$ acting on $M$ are just $\{e\}$ and $K$
since the isotropy groups of closed orbits of $G$ acting on $V$ are just
$\{e\}$ and $G$. Let $1\leq k\leq n-2$ and consider the point $x=((e_1,\dots,e_k,0,\dots,0),(0,\dots,0,e_1^\ast,\dots,e_k^\ast))\in V\oplus V^*$. Here the $e_i$ and $e_i^\ast$ are the usual basis and dual basis of $\C^n$. Then $x\in N$ lies on a closed $G$-orbit with isotropy group $\SL_{n-k}(\C)$.
\end{example}

% xxxxxxxxxxxxxxxxxxxxxxxxxxxxxxxxxxxxxxxxxxxxxxxxxxxxxxxxxxxxxxxxxxxxxxxxx

\subsection{Symplectic singularities and graded Gorenstein rings}
\label{subsec:BackSympSing}

Recall the following.

\begin{definition}[Beauville, \cite{Beauville}]
\label{def:SymplecticSing}
An algebraic variety $X$ over $\C$ has a \emph{symplectic singularity} at $p \in X$ if there
is an open neighborhood $U$ of $p$ (classical topology) such that
\begin{enumerate}
\item[({\it i})]   $U$ is normal,
\item[({\it ii})]   the smooth locus $U_{\sm}$ of $U$ admits a holomorphic symplectic (i.e., closed and
        non-degenerate) $2$-form $\omega$, and
\item[({\it iii})]   for any resolution $f\co Y \to U$ of singularities,
$(f|_{f^{-1}(U_{\sm})})^\ast\omega$
        extends to a holomorphic $2$-form on $Y$.
\end{enumerate}
We say that $X$ has \emph{symplectic singularities} or is a \emph{symplectic variety} if it
has a symplectic singularity at each point.
\end{definition}

Note that in condition ({\it iii}) of this definition, if there is a resolution $f\co Y\to U$ such that
$(f|_{f^{-1}(U_{\sm})})^\ast\omega$ extends to a symplectic form on $Y$, then $Y$ is called a
\emph{symplectic resolution} of $U$. Symplectic varieties need not admit symplectic resolutions.
See \cite{FuSurveySympSing} for a survey of symplectic singularities and resolutions.

As noted in \cite{Beauville}, the main theorem of \cite{FlennerExtend} implies that
condition ({\it iii}) of Definition \ref{def:SymplecticSing} is always satisfied
if the variety is smooth in codimension $3$; that is, we have the following.

\begin{theorem}[Flenner, \cite{FlennerExtend}]
\label{thrm:FlennerCodim}
Let $X$ be a normal variety such that $X\smallsetminus X_{\sm}$ has codimension at least $4$.
Then $X$ has symplectic singularities if and only if $X_{\sm}$ admits a holomorphic symplectic form.
\end{theorem}

Similarly, using the results of \cite[Proposition (1.2), Remark (1.3), and Theorem 7]{Reid},
Beauville demonstrates \cite[Proposition 1.3]{Beauville} that symplectic varieties are
Gorenstein with rational singularities. A converse is proven in \cite[Theorem 6]{NamikawaExtension}, yielding
the following.

\begin{theorem}[Namikawa, \cite{NamikawaExtension}]
\label{thrm:NamikawaRatGoren}
A normal variety $X$ has symplectic singularities if and only if it is
Gorenstein with
rational singularities and the smooth locus $X_{\sm}$ admits a holomorphic symplectic form.
\end{theorem}

Let $R$ be an $\mathbb{N}$-graded algebra
of dimension $d$ such that $R_0$ is equal to the ground field.
Recall \cite[Definition 2.1.2]{GotoWatanabe} that the \emph{canonical module} of $R$ is
defined to be the dual of the $d$th local cohomology group $\underline{H}_{\mathfrak{m}}^d(R)$
where $\mathfrak{m}$ denotes the maximal homogeneous ideal; see Section \ref{subsec:GradGoren}
for an explicit description. An algebraic variety $X$ is Cohen-Macaulay if and only if its
\emph{dualizing complex} is a sheaf, corresponding to the canonical module, and is Gorenstein if and only if this sheaf
is invertible, see \cite[Note 3.5.9, Proposition 3.5.12, and Definition 6.2.1]{Ishii} and
\cite[\href{https://stacks.math.columbia.edu/tag/0A7A}{Tag 0A7A},
\href{https://stacks.math.columbia.edu/tag/0DW4}{Tag 0DW4},
\href{https://stacks.math.columbia.edu/tag/0DW6}{Tag 0DW6}]{stacks-project}.
Recall further that $R$ is Gorenstein if and only if it is graded isomorphic to its canonical module,
possibly with a shift in the grading, see \cite[Section 21.11]{EisenbudComAlg} \cite[Proposition 2.1.3]{GotoWatanabe}.
This shift is given by $-(a(R) + d)$ where $a(R)$ is the \emph{$a$-invariant} of $R$
\cite[Definition 3.6.13]{BrunsHerzog}, \cite[Section 3]{GotoWatanabe}. If $a(R)$ is equal to the negative Krull dimension
of $R$, then we say that $R$ is \emph{graded Gorenstein}, also called
\emph{strongly Gorenstein} \cite[page 103]{DerskenKemperBook}.

By \cite[Theorem 4.4]{StanleyHilbFunGradAlg}, for a Cohen-Macaulay ring $R$,
the Gorenstein and graded Gorenstein properties can be determined from the Hilbert
series of $R$. Specifically, we have the following.

\begin{theorem}[Stanley, \cite{StanleyHilbFunGradAlg}]
\label{thrm:StanleyGorenHilb}
Let $R$ be a graded commutative Noetherian algebra over a field $k = R_0$.
Suppose that $R$ is a Cohen-Macaulay integral domain with Krull dimension $d$.
Then $R$ is Gorenstein if and only if the Hilbert series $\Hilb_R(t)$ satisfies
\begin{equation}
\label{eq:StanleyGorenHilb}
    \Hilb_R(1/t) = (-1)^d t^{-a} \Hilb_R(t)
\end{equation}
for an integer $a$, which is then equal to $a(R)$.
\end{theorem}

If $G$ is finite, then the moment map is trivial so that the complex symplectic quotient is the usual quotient
$(V\oplus V^\ast)/G$. We refer to a quotient of this form as a \emph{complex symplectic orbifold};
note that we consider orbifolds as varieties with quotient singularities rather than
the finer structures of Deligne-Mumford stacks or Lie groupoids.
Because $G$ acts as a subgroup of $\SL(V\oplus V^\ast)$, we have the following consequence of
\cite[Proposition 2.4]{Beauville} and \cite[Theorem 1]{WatanabeGorenstein}.

\begin{theorem}[Beauville \cite{Beauville} and Watanabe \cite{WatanabeGorenstein}]
\label{thrm:Orbifolds}
Suppose that $G$ is a finite group and $V$ is a $G$-module. Then the complex symplectic orbifold
$(V\oplus V^\ast)/G$ has symplectic singularities and is graded Gorenstein.
\end{theorem}

The fact that $(V\oplus V^\ast)/G$
has rational singularities follows from a theorem of Boutot \cite{Boutot},
which applies for any reductive $G$.

% xxxxxxxxxxxxxxxxxxxxxxxxxxxxxxxxxxxxxxxxxxxxxxxxxxxxxxxxxxxxxxxxxxxxxxxxx
% xxxxxxxxxxxxxxxxxxxxxxxxxxxxxxxxxxxxxxxxxxxxxxxxxxxxxxxxxxxxxxxxxxxxxxxxx
% xxxxxxxxxxxxxxxxxxxxxxxxxxxxxxxxxxxxxxxxxxxxxxxxxxxxxxxxxxxxxxxxxxxxxxxxx

\section{Auxiliary results}
\label{sec:AuxResults}

% xxxxxxxxxxxxxxxxxxxxxxxxxxxxxxxxxxxxxxxxxxxxxxxxxxxxxxxxxxxxxxxxxxxxxxxxx

\subsection{Normality of the shell}
\label{subsec:NormalityShell}

By \cite[Theorem 2.2(4)]{HerbigSchwarz}, if the $G$-module $V$ is $2$-large, then the shell $N$ is a normal
variety. However, this condition is sufficient but not necessary: there are $G$-modules $V$ that are not $2$-large
such that $N$ is normal.
For example, letting $V \simeq \C^n$ denote the defining representation of
$\SO_n(\C)$, the module
$(V^{\oplus n}, \SO_n(\C))$ is $2$-large while the module $(V^{\oplus n}, \OO_n(\C))$ is not.
Because the shell coincides for both cases, it is normal \cite[page 17]{CapeHerbigSeaton}. See also
Section \ref{sec:Torus} below.
In fact, using the isomorphism of $\C[N]$ with the symmetric algebra of the Jacobian module
described in \cite[Remark 2.4]{HerbigSchwarz},
assuming $V_{(0)}\neq\emptyset$, the condition that $(V,G)$ is
$2$-modular is equivalent to the shell $N$ being factorial by \cite[Proposition 6]{Avramov}.
\begin{remark}
\label{rem:Factorial}
If the abelianization $G/[G,G]$ of $G$ is trivial and $(V,G)$ is $3$-large, then
Bellamy and Schedler have proven that the complex symplectic quotient $N\git G$
is locally factorial \cite[Theorem 1.2(c)]{BellamySchedler}.
This can also be established using the above fact that $N$ is in this case factorial
along with \cite[Theorem 8.4]{Drezet}.
In particular, the principal isotropy group is the kernel of the $G$-action
on $V\oplus V^\ast$ \cite[Corollary 7.7(2)]{GWSlifting} so that by taking the quotient,
we may assume that $(V,G)$ has TPIG. Take the saturated open set in
\cite[Theorem 8.3]{Drezet} to be the union of the principal orbits. As $N$ is factorial,
$\operatorname{Pic}(N)$ is trivial by \cite[Proposition 3.5.1]{BensonBook}.
It remains only to show that the units $\C[N]^\ast$ of $\C[N]$ are constant. However,
$N$ is a cone and hence stable under the scalar action of $\C^\times$. It is
verified in the proof of \cite[Proposition, Section 1.3]{KnopKraftVust} that
$\C^\times$ acts on each element $f \in \C[N]^\ast$ as multiplication by a character $\chi$,
i.e. $f\circ t = \chi(t) f$ for $t \in \C^\times$. Then $f$ must be homogeneous
and cannot vanish at the origin, implying that $f$ is constant.
\end{remark}

In this section, we give the following characterization of $G$-modules $V$ such that $N$ is normal; compare the
similar conditions of \cite[Lemma 3.3]{KaledinLehnSorger}
as well as the related smoothness criterion in
\cite[Lemma 2.1]{DoranHoskins}.

\begin{proposition}
\label{prop:NormalityShell}
Suppose that $(V, G)$ is $0$-modular, and
let $R = \{(x,\xi)\in V\oplus V^\ast \mid G_{(x,\xi)}$ is finite$\}$.
\begin{enumerate}
\item[({\it i})]    The shell $N$ is a complete intersection, hence $\C[V\oplus V^\ast]/(\mu)$
                    is Cohen-Macaulay.
\item[({\it ii})]   The shell $N$ is reduced and irreducible if and only if $V$ is $1$-modular.
\item[({\it iii})]   If $V$ is $2$-modular, then $N$ is normal.
\item[({\it iv})]   The set of smooth points $N_{\sm}$ of $N$ is the set of points
    in $N$ on which $d\mu$ has maximal rank.
\item[({\it v})]   The set $R$ is the set of points in $V\oplus V^\ast$ on which $d\mu$ has maximal rank.
\item[({\it vi})]   The shell $N$ is normal if and only if $N\smallsetminus (N\cap R)$ has complex codimension at least two in $N$.
\end{enumerate}
\end{proposition}
\begin{proof}
As $(V, G)$ is $0$-modular,
({\it i}) follows from \cite[Proposition 9.4]{GWSlifting}. In particular,
$N$ is equidimensional.

For each $r\geq 0$, let $N_r$ denote the closure of $(V_{(r)}\times V^\ast)\cap N$,
and note that each $(V_{(r)}\times V^\ast)\cap N$ is a vector bundle over $V_{(r)}$ with
fiber dimension $\dim_\C V - \dim_\C G + r$.
Because $(V,G)$ is $0$-modular, $\dim_\C V_{(r)} \leq \dim_\C V - r$ for each $r$.
Assume $(V, G)$ is not $1$-modular, and then there is an $r > 0$ such that $\dim_\C V_{(r)} = \dim_\C V - r$,
implying $\dim_\C N_r = \dim_\C N_0$. But then $N_r$ is not a subset of $N_0$, implying that $N$ is
not irreducible.
Now assume $(V, G)$ is $1$-modular, and then for each $r \geq 1$,
$\dim_\C V_{(r)} \leq \dim_\C V - r - 1$ so that $\dim_\C N_r < \dim_\C N_0$. As $N$
is equidimensional, it must then be that $N_r\subset N_0$ for each $r\geq 1$. Hence $N$ is equal
to the closure of $N_0 \smallsetminus\bigcup_{r\geq 1} N_r$. As $V_{(0)}$ is irreducible,
and as $(V_{(0)}\times V^\ast)\cap N$ is a complex vector bundle over $V_{(0)}$, its closure
$N_0$ is irreducible. Then
$N_0 \smallsetminus\bigcup_{r\geq 1} N_r$ is as well irreducible and smooth, hence reduced,
proving ({\it ii}). To prove ({\it iii}), note that if $(V, G)$ is $2$-modular,
then the $N_r$ have
complex codimension at least $2$ in $N$. Then $N$ is smooth in codimension $1$
and Cohen-Macaulay by ({\it i}), hence normal by Serre's criterion \cite[Theorem 23.8]{MatsumuraBook}.

As $N$ is a complete intersection, ({\it iv}) is simply the Jacobian criterion.
Now, for each $A \in \mathfrak{g}$, let $\mu^A\co V\oplus V^\ast \to \C$ denote the
evaluation of $\mu$ at $A$. As $\mu$ is bilinear, $d\mu^A(x,\xi) = 0$ if and only if
$A(x) = \xi\circ A = 0$.    Then ({\it v})
follows from the fact that $R$ is the set of points where  no $d\mu^A$ vanishes for   $A \in\mathfrak{g}\smallsetminus\{0\}$,
and ({\it vi}) follows from the fact that $N$ is then smooth in codimension
$1$ if and only if
$N\smallsetminus (N\cap R)$ has codimension at least $2$ in $N$.
\end{proof}
See \cite[Remark 2.4]{HerbigSchwarz} for an alternate proof of ({\it ii}) and ({\it iii}).

% xxxxxxxxxxxxxxxxxxxxxxxxxxxxxxxxxxxxxxxxxxxxxxxxxxxxxxxxxxxxxxxxxxxxxxxxx

\subsection{The $k$-large property for $V^\ast$}
\label{subsec:KLargeDual}

Let $G$ be a reductive group and $V$ a $G$-module. The goal of this section is to demonstrate that
several properties of $V$ imply the same for the dual module $V^\ast$. We collect these in the following.

\begin{theorem}
\label{thrm:KLargeDual}
Let $G$ be a reductive group and $V$ a $G$-module.
\begin{enumerate}
\item[({\it i})]   The $G$-module $V$ is stable if and only if $V^\ast$ is stable.
\item[({\it ii})]   The $G$-module $V$ has TPIG (respectively FPIG) if and only if $V^\ast$ has TPIG (respectively FPIG).
\item[({\it iii})]   The $G$-module $V$ is $k$-modular if and only if $V^\ast$ is $k$-modular.
\item[({\it iv})]    The $G$-module $V$ is $k$-principal if and only if $V^\ast$ is $k$-principal.
\item[({\it v})]   The $G$-module $V$ is $k$-large if and only if $V^\ast$ is $k$-large.
\end{enumerate}
\end{theorem}

Let $\Cl(V)$ denote the set of conjugacy classes $(G_v)$ of isotropy groups of closed orbits $Gv$ in $V$. For $(H)\in \Cl(V)$ let $(V\git G)_{(H)}$ denote the corresponding stratum of $V\git G$ and let $V^{(H)}$ denote the inverse image of $(V\git G)_{(H)}$ in $V$. We need the following lemma.

\begin{lemma}
\label{lem:Cl(V)}
We have $\Cl(V)=\Cl(V^\ast)$ and for each $(H)\in\Cl(V)$  we have
$\dim_\C (V\git G)_{(H)}=\dim_\C (V^\ast\git G)_{(H)}$
and $\dim_\C V^{(H)}=\dim_\C (V^\ast)^{(H)}$.
\end{lemma}
\begin{proof}
We follow ideas of the proof of \cite[Proposition E]{Homaloidal}.
First assume that $G$ is connected. Then there is an automorphism $\tau$ of $G$ which interchanges $W$ and $W^\ast$ for every
irreducible $G$-module $W$. Specifically, $\tau$ induces an automorphism
of the weights of $G$-modules sending the highest weight of $W$ to that of $W^\ast$. Thus we may consider $V^\ast$ as a representation of $G$ on the same space $V$ but with $g\in G$ acting via $\tau(g)$. Then
$\C[V]^G=\C[V^\ast]^G$ and for each isotropy group $H$, the identification sends
$V^{(H)}$ to $(V^\ast)^{(\tau(H))}$ and sends $(V\git G)_{(H)}$ to $(V^\ast\git G)_{(\tau(H))}$.

Returning to the general case so that $G$ need not be connected, let $(H)\in\Cl(V)$. Then by \cite[Lemma 5.5]{GWSliftingHomotopies}, $V^H$ has a Zariski open subset  consisting of the intersection of $V^H$ with  closed orbits with isotropy group $H$. Let   $V^{\langle H\rangle}$ denote this intersection. Then  the image of $V^{\langle H\rangle}$ in $V\git G$  is $(V\git G)_{(H)}$. Because $H$ is an isotropy group of $V$, the action of $N_G(H)/H$ on $V^H$ has TPIG,
and $\dim_{\C} V^H\git N_G(H) = \dim_{\C} (V\git H)_{(H)}$. By \cite[Proposition E(2)]{Homaloidal}, the action of
$N_G(H)/H$ on $(V^H)^\ast\simeq (V^\ast)^H$ has TPIG as well. By \cite[Corollaire 1]{LunaAdherences},
a $G$-orbit through $x$ in $(V^\ast)^H$ is closed if and only if $(N_G(H)/H)x$ is closed. Hence there is a
Zariski open subset $U^\ast_H$ of $(V^\ast)^H$ consisting of points with closed $G$-orbit.
Then there is an $H^\prime$ with $(H^\prime)\in\Cl(V^\ast)$ such that
$U^\ast_H$ has an open intersection with $(V^\ast)^{\langle H^\prime\rangle}$ where $H\subset H^\prime$, $\dim_\C H^\prime=\dim_\C H$, and $(V^\ast)^{H^\prime}\subset (V^\ast)^H$. Then as $(V^\ast)^H\simeq (V^H)^\ast$ and similarly for $H^\prime$, $(V^\ast)^{H^\prime}=(V^\ast)^H$ and $V^H=V^{H^\prime}$. Applying the same argument to $(H^\prime)\in\Cl(V^\ast)$, we see that $H\subset H^\prime\subset H^{\prime\prime}$ where $H^{\prime\prime}$ is the isotropy group of closed $G$-orbits intersecting an open subset of $V^{H^{\prime\prime}}=V^{H^\prime}=V^H$. Hence $H=H^{\prime\prime}=H^\prime$ and $\dim_\C (V\git G)_{(H)}=\dim_\C (V^\ast\git G)_{(H)}$.

Let $(W\oplus \C^d)$ be the slice representation of $H$ where $W^H=0$ and $H$ acts trivially on $\C^d$. Then  $V\simeq \mathfrak g/\mathfrak h\oplus W\oplus \C^d$ as $H$-module. Letting $e = \dim_\C (V\git G)_{(H)}$, the dimension of $V^{(H)}$ is $\dim_\C \mathfrak g/\mathfrak h+e+\dim_\C\NN(W)$ where $\NN(W)$ is the null cone of $W$. For the action of $H$ on $V^\ast$ we get the decomposition $V^\ast\simeq\mathfrak g/\mathfrak h\oplus W^\ast\oplus \C^e$, so it remains only to show that $\dim_\C\NN(W)=\dim_\C\NN(W^\ast)$. But as a set, the null cone of $W$ is determined by the action of $H^0$, and the identification
of the $H^0$-modules $W$ with $W^\ast$ such that $\C[W]^{H^0}=\C[W^\ast]^{H^0}$ identifies the null cones. Thus $\dim_\C V^{(H)}=\dim_\C (V^\ast)^{(H)}$.
\end{proof}

\begin{remark}
\label{rem:RegSeqMoment}
Let $V$ be a $G$-module with FPIG.
Choose a basis for $\mathfrak{g}$ and let $\mu_1,\ldots,\mu_m$ denote the component functions
of $\mu$ with respect to this basis.
It can be shown that in this context, the $k$-large property for $(V,G)$ is
equivalent to the existence of homogeneous $f_1,\dots,f_k\in\C[V]^G$, vanishing on the inverse images in $V$ of
the non-principal strata of $V\git G$, such that  $f_1,\ldots,f_k,\mu_1,\ldots,\mu_m$ is a regular sequence in
$\C[V\oplus V^\ast]$. See \cite[8.2 and Lemma 9.7]{GWSlifting}.
We will need this fact in Proposition \ref{prop:3LargeCodim} below.
\end{remark}

\begin{proof}[Proof of Theorem \ref{thrm:KLargeDual}]
We first note that ({\it i}) is proven in \cite[Proposition E]{Homaloidal}, where it is also demonstrated
that the principal isotropy groups of $V$ and $V^\ast$ are the same. Then ({\it ii}) is an immediate
consequence. Since $(V,G)$ is $k$-modular if and only if $(V,G^0)$ is $k$-modular, ({\it iii}) is immediate using the automorphism $\tau$ of Lemma \ref{lem:Cl(V)}.
Statement ({\it iv}) follows directly from Lemma \ref{lem:Cl(V)} and, of course,  ({\it v}) follows from  ({\it ii}),  ({\it iii}) and  ({\it iv}).
\end{proof}

% xxxxxxxxxxxxxxxxxxxxxxxxxxxxxxxxxxxxxxxxxxxxxxxxxxxxxxxxxxxxxxxxxxxxxxxxx

\subsection{The $k$-large property is generic}
\label{subsec:MostVLarge}

In this section, we support the claim that ``most" $G$-modules are $3$-large by demonstrating the following.

\begin{theorem}
\label{thrm:MostVLarge}
Let $G$ be connected and semisimple and let $k\in\N$. Among $G$-modules $V$ such that $V^G = \{0\}$
and every irreducible component of $V$ is a faithful $\mathfrak{g}$-module, there are up to isomorphism
only finitely many $V$ that are not $k$-large.
\end{theorem}

This theorem was established for the case $k = 2$ in \cite[Corollary 11.6(2)]{GWSlifting}.
Note that for $G$ simple, Theorem \ref{thrm:MostVLarge} implies that all but finitely
$G$-modules $V$ such that $V^G = \{0\}$ are $k$-large.

To begin, let $Y(G)$ denote the set of $1$-parameter subgroups of $G$. Let $\lambda\in Y(G)$ and let $V_\lambda^+$
(respectively $V_\lambda^-$) denote the subspace of $V$ on which $\lambda$ has strictly positive
(respectively strictly negative) weights. Similarly let $V_\lambda^0$ be the subspace of $V$
fixed by $\lambda$. Then by \cite[Lemma 6]{Go}, there is a constant $d\geq 2$,
depending only on $G$, such that
\begin{equation}
\label{eq:MostGoEst1}
    \dim_{\C} V_\lambda^- \geq \frac{\dim_{\C} (V_\lambda^+ \oplus V_\lambda^-)}{d}.
\end{equation}
Then we have the following.

\begin{corollary}
\label{cor:MostSliceBound}
Let $v \in V$ have closed $G$-orbit, let $H = G_v$, and assume $\dim_{\C} H > 0$.
Let $W$ be the slice representation at $v$
and let $\lambda\in Y(H)$. Then
\[
    \dim_{\C} \big(W_\lambda^- \oplus (\mathfrak{g}/\mathfrak{h})_\lambda^-\big)
        \geq    \frac {\dim_{\C} (W_\lambda^+ \oplus W_\lambda^-)}{d}.
\]
\end{corollary}
\begin{proof}
This follows from Equation \eqref{eq:MostGoEst1} and the fact that
$V\simeq \mathfrak{g}/\mathfrak{h}\oplus W$ as $H$-module.
\end{proof}

We also have the following bound on the codimension of the null cone
$\NN(W)$ of $W$. Let $\widehat{W}$ be an $H$-complement of $W^H$ so that $W = W^H\oplus \widehat{W}$.
Here, we use the notation $\codim(A,B)$ for the codimension of
$A$ in $B$ to avoid ambiguities.

\begin{lemma}
\label{lem:MostNullBound1}
Suppose that $\dim_{\C} H > 0$. Then
\begin{equation}
\label{eq:MostNullFirst}
    \codim_\C\big(\NN(\widehat{W}),\widehat{W}\big)
        \geq \min_{\lambda\in Y(H)}\left\{ \dim_{\C} (W_\lambda^- \oplus \widehat{W}_\lambda^0)\right\} - \dim_{\C} U
\end{equation}
where $U$ is a maximal unipotent subgroup of $H$.
\end{lemma}
Note that $\NN(\widehat{W}) = \NN(W)$.
\begin{proof}
Let $T$ be a maximal torus of $H$ and let $Y^\prime(T)$ denote the set of $\lambda\in Y(T)$
which are in the positive Weyl chamber.
Note that for $\lambda\in Y(H)$, $\widehat{W}_\lambda^{\pm} = W_\lambda^{\pm}$.
By the Hilbert-Mumford criterion \cite[Theorem 5.2]{PopovVinberg},
there are finitely many $\lambda\in Y^\prime(T)$ such that $\NN(\widehat{W})$
is the union of the $H W_\lambda^+$. As each $W_\lambda^+$ is stabilized by a Borel
subgroup $B$ of $H$, $\dim_{\C} H W_\lambda^+ = \dim_{\C} U W_\lambda^+$
where $U$ is the unipotent subgroup opposite to the unipotent radical of $B$.
Then the fact that
$\codim_{\C}(W_\lambda^+, \widehat{W}) = \dim_{\C} (W_\lambda^-\oplus \widehat{W}_\lambda^0)$
completes the proof.
\end{proof}

Using Corollary \ref{cor:MostSliceBound}, we can rewrite Equation \eqref{eq:MostNullFirst}
in Lemma \ref{lem:MostNullBound1} as
\begin{align*}
    \codim_\C\big(\NN(\widehat{W}),\widehat{W}\big)
        &\geq
            \min_{\lambda\in Y(H)}\left\{
                \frac {\dim_{\C} (W_\lambda^+ \oplus W_\lambda^-)}{d}
                + \dim_{\C} \widehat{W}_\lambda^0
                - \dim_{\C} U - \dim_{\C} (\mathfrak{g}/\mathfrak{h}) \right\}
        \\ &\geq
            \frac {\dim_{\C} \widehat{W}}{d} - \dim_{\C} U - \dim_{\C} \mathfrak{g}.
\end{align*}
Because $V \simeq W^H\oplus \widehat{W}\oplus (\mathfrak{g}/\mathfrak{h})$ as
$H$-module, we have
\[
    \dim_{\C}\widehat{W}
        =           \codim_{\C}(W^H, V) - \dim_{\C} (\mathfrak{g}/\mathfrak{h})
        \geq        \codim_{\C}(V^H, V) - \dim_{\C} (\mathfrak{g}/\mathfrak{h}).
\]
By \cite[Proposition 6]{Go}, as soon as $\dim_{\C} V \geq \dim_{\C} G$, we have for any $L \leq G$
with nontrivial action on $V$ that
\begin{equation}
\label{eq:MostGoEst2}
    \codim_{\C}(V^L, V)
        \geq    \frac{\sqrt{\dim_{\C} V}}{17}.
\end{equation}
Combining these observations with $L = H$ yields the following.

\begin{corollary}
\label{cor:MostNullBound2}
If $\dim_{\C} V \geq \dim_{\C} G$, then
\[
    \codim_\C\big(\NN(\widehat{W}),\widehat{W}\big)
        \geq \frac {\sqrt{\dim_{\C} V}}{17d} - \dim_{\C} U
        - 2\dim_{\C} \mathfrak{g}.
\]
\end{corollary}

With this, we proceed with the proof of the main result of this section.

\begin{proof}[Proof of Theorem \ref{thrm:MostVLarge}]
Let $v \in V$ have closed $G$-orbit, let $H = G_v$, let $W$ be the slice representation at $v$, and
write $W = W^H\oplus \widehat{W}$ as above. Recall from Section \ref{subsec:KLargeDual} that $(V\git G)_{(H)}$ denotes
the stratum of points with isotropy group conjugate to $H$ and $V^{(H)}$ denotes its inverse image in $V$.
It follows from Luna's slice theorem \cite{LunaSlice} that the codimension of $V^{(H)}$ in $V$ is equal
to the codimension of the null cone $\NN(\widehat{W})$ in $\widehat{W}$. If $H$ is finite and nontrivial,
then this is just the codimension of $W^H$ in $W$, which is bounded below by the codimension of $V^H$ in $V$
minus the dimension of $\mathfrak{g}$.
Applying Equation \eqref{eq:MostGoEst2}, we have that
$\codim_{\C} V^{(H)} \to\infty$ as $\dim_{\C} V \to \infty$.
If $\dim_{\C} H > 0$, then $\codim_{\C}(\NN(\widehat{W}),\widehat{W})\to\infty$ as $\dim_{\C}V\to\infty$
by Corollary \ref{cor:MostNullBound2} and Equation \eqref{eq:MostGoEst2}.

Now, note that $V\smallsetminus V_\pr$ is the union of the $V^{(H)}$ for non-principal isotropy
types $(H)$ corresponding to closed orbits. The estimates above imply that $V$ is $k$-principal
for $\dim_{\C} V$ sufficiently large. Similarly, for all but finitely many $V$, each $V^{(H)}$
with $H$ infinite has positive codimension. As the complement of the $V^{(H)}$ consists of points
with closed orbit and finite isotropy, all but finitely many $V$ have FPIG.
Then by the proof of \cite[Corollary 11.6]{GWSlifting},
replacing the inequality $\kappa_m(V) \geq 2$ with $\kappa_m(V) \geq k$,
there are up to isomorphism only finitely many $V$ which
are not $k$-modular. Thus all but finitely many V are $k$-large.
\end{proof}

% xxxxxxxxxxxxxxxxxxxxxxxxxxxxxxxxxxxxxxxxxxxxxxxxxxxxxxxxxxxxxxxxxxxxxxxxx

\subsection{The symplectic slice theorem for complex symplectic quotients}
\label{subsec:SympSlice}

In this section, we adapt Luna's \'etale slice theorem \cite{LunaSlice} to the context of
the Hamiltonian actions considered in this paper.
We use this formulation to demonstrate that for $1$-large $(V,G)$
such that $N$ is normal and $(N\git G)_{\pr} = (N\git G)_{\sm}$, the symplectic form on
$V\oplus V^\ast$ induces a symplectic form on the smooth points of the complex symplectic quotient.
This strengthens the criteria given in \cite[Proposition 2.4]{Becker}; see also
\cite[Lemma 2.5]{DoranHoskins}.
The results of this section are closely related to those of \cite[Section 6]{BuloisLehnTerpereau}. We give a slightly more explicit description of the shell in the slice representation in order to understand how the slice representation
inherits properties
such as $1$-modular from $(V,G)$, and as well to determine bounds on the codimensions of the non-principal points.

We begin with the following observations regarding the structure of $G$-modules in the
presence of invariant bilinear forms.

\begin{lemma}
\label{lem:WWstar}
Let $V$ be a $G$-module which admits non-degenerate $G$-invariant bilinear forms $\omega$ and $\sigma$
such that $\omega$ is skew-symmetric and $\sigma$ is symmetric. Then $V\simeq W\oplus W^\ast$ for some
$G$-module $W$, and $W$ and $W^\ast$ are isotropic relative to $\omega$.
\end{lemma}
\begin{proof}
Define the linear map $\tau\co V\to V$ by $\omega(v,v^\prime)=\sigma(\tau(v),v^\prime)$ for $v$, $v^\prime\in V$.
In coordinates, if $\Omega$ and $\Sigma$ denote the matrices of $\omega$ and $\sigma$, respectively, then
$\tau$ is given by $-\Sigma^{-1}\Omega$.
Hence $\tau$ is $G$-equivariant and skew-symmetric.
The eigenvalues of $\tau$ are in pairs of the form $\pm \lambda i$ for $\lambda>0$. Thus we can write
$V = V_+\oplus V_-$ where $V_+$ (respectively $V_-$) is the sum of the eigenspaces for positive
(respectively negative) multiples of $i$. Then $\omega$ vanishes on $V_+$ and $V_-$ and therefore induces
a dual pairing of $V_+$ and $V_-$, hence $V\simeq V_+\oplus (V_+)^\ast$.
\end{proof}

\begin{corollary}
\label{cor:WWstar}
Let $V$ be a $G$-module and let $U$ be a $G$-submodule such that
$\widetilde V:=V\oplus V^\ast$ contains a submodule of the form $\widetilde U\simeq U\oplus U^\ast$.
Then there is a $G$-module $W$ and a $G$-submodule $\widetilde W\subset \widetilde V$ of the form
$\widetilde W\simeq W\oplus W^\ast$ such that $\widetilde V =\widetilde U\oplus \widetilde W$.
\end{corollary}
\begin{proof}
We follow an argument from \cite[Lemma 5]{LunaClosed}. Let $\omega_1$ be a non-degenerate $G$-invariant
bilinear skew-symmetric form on $\widetilde V$ and let $\omega_2$ be a non-degenerate $G$-invariant bilinear
skew-symmetric form on $\widetilde U$. We extend $\omega_2$ by zero on a $G$-complement to $\widetilde U$ in
$\widetilde V$ to consider it as a form on $\widetilde{V}$. Consider the $G$-invariant skew-symmetric
forms $\omega_t:=t\omega_1+(1-t)\omega_2$ for $t\in \R$.
For most $t$, the form $\omega_t$ is non-degenerate on both $\widetilde V$ and $\widetilde U$. Let $t_0$ be such
a $t$, and let $\widetilde W$ denote the orthogonal complement to $\widetilde U$ with respect to $\omega_{t_0}$.
Then $\omega_{t_0}$ is non-degenerate on $\widetilde W$ and induces a non-degenerate form on
$\widetilde V/\widetilde U\simeq\widetilde W$. The same argument shows that  $\widetilde V/\widetilde U$
admits a non-degenerate  $G$-invariant symmetric bilinear form, hence so does $\widetilde W$. Now we can
apply Lemma \ref{lem:WWstar}.
\end{proof}

Recall that if  $X$ and $Y$ are affine $G$-varieties, a $G$-morphism $\phi \co X\to Y$
is said to be \emph{excellent\/} if
\begin{enumerate}
\item[({\it i})] $\phi$ is \'etale,
\item[({\it ii})] the induced morphism $\phi\git G : X\git G \to Y\git G$ is \'etale, and
\item[({\it iii})] the morphism $(\phi,\pi_X) : X \to Y \times_{Y\git G} X\git G$ is an isomorphism.
\end{enumerate}
Thinking analytically, $\phi$ is \'etale if and only if it induces a locally biholomorphic map
of the complex analytic varieties associated to $X$ and $Y$. In particular, if $X$ and $Y$ are
smooth, $\phi$ is \'etale if and only if the differential of $\phi$ is an isomorphism at each
point of $X$.

\begin{lemma}
\label{lem:ExcellentPullback}
Let $X$ and $Y$ be smooth affine $G$-varieties. Let $\omega$ be a symplectic form on $Y$ with respect
to which $Y$ is Hamiltonian with moment map $\mu\co Y\to\mathfrak{g}^\ast$. If
$\phi\co X\to Y$ is excellent, then $X$ is Hamiltonian with respect to the symplectic form
$\phi^*\omega$ and moment map $\phi^*\mu$.
\end{lemma}
\begin{proof}
For each $x\in X$ and $A\in\mathfrak g$, one checks that $d((\phi^*\mu)^A)(x)=(\iota_A\phi^*\omega)(x)$.
Hence $\phi^*\mu$ is a moment mapping for the symplectic action of $G$ on $X$.
\end{proof}

Let $X$ be an affine $G$-variety, and let $x \in X$. Recall that an \'etale slice at $x$ is a
$G_x$-invariant locally closed subvariety $S^\prime$ of $X$ containing $x$ such that the morphism $G\times_{G_x} S^\prime\to X$
sending $[g,s]\mapsto gs$ is excellent. Here $G\times_{G_x} S^\prime$ is the quotient of $G\times S^\prime$ by the free
$G_x$-action $h(g,s)=(gh^{-1},hs)$ for $h\in G_x$, $g\in G$, $s\in S^\prime$. Then $[g,s]$ denotes the image of $(g,s)$ in $G\times_{G_x} S^\prime$. By Luna's \'etale slice theorem \cite{LunaSlice}, there is a
slice at each point $x$ such that the orbit $Gx$ is closed.

We will now describe the form of an \'etale
slice in the presence of a $G$-invariant symplectic structure at a point in the shell.
Let $(V,K)$ be unitary, let $x\in N \subset V\oplus V^\ast$
with $Gx$ closed,
set $L = G_x$, and let $\mathfrak{l} \subset \mathfrak{g}$ denote the Lie algebra of $L$.
Let $\perp$ denote
orthogonal with respect to the symplectic form. Since $x\in N$, $E:=T_x(Gx)$ is isotropic and we have a direct sum decomposition of $L$-modules
\[
    V\oplus V^\ast\simeq S\oplus E\oplus E^*
\]
where $S\oplus E=E^\perp$
 and
$E^*\simeq(\mathfrak g/\mathfrak l)^\ast$ is paired dually with $E\simeq \mathfrak g/\mathfrak l$
via $\omega$. Then $\omega$ restricts to a  symplectic form $\omega_S$ on $S$. We call $(S,L)$ the \emph{symplectic slice representation at $x$}.
We will show that,  up to an  excellent morphism,
$N$ near $x$ is isomorphic to $G\times_LN_S$ at $[e,0]$ where
$N_S$ is the zero set of the  moment map $\mu_S\co S\to\mathfrak l^*$.

\begin{proposition}
\label{prop:SympSlices}
Let $V$ be a $G$-module  and let $\omega$ denote the standard symplectic form on $V\oplus V^\ast$.
Let $x=(v,\xi)\in N$ such that $Gx$ is closed
with isotropy group $L$, let $(S,L)$ denote the symplectic slice representation at $x$
and let $\omega_S$ denote the restriction of $\omega$ to $S$.
Then   $S = W\oplus W^\ast$ where $W$ is an $L$-module and
$W$ and $W^\ast$ are isotropic subspaces relative to $\omega_S$.
\end{proposition}
\begin{proof}
By Corollary
\ref{cor:WWstar}, $S\simeq W\oplus W^\ast$ where $W$ is an $L$-module. Since $\omega_S$ is non-degenerate,
$W$ and $W^\ast$ can be chosen to be isotropic by Lemma \ref{lem:WWstar}.
\end{proof}

Because $\omega$ is $G$-invariant, as an $L$-module, $V\oplus V^\ast$ is the orthogonal symplectic direct
sum $(E\oplus E^\ast) \oplus (W\oplus W^\ast)$.
For the $G$-action on $V\oplus V^\ast$, the slice representation $S^\prime$ at $x$ is $E^\ast\oplus S$.
We have the canonical mapping
\[
    \phi\co G\times_L(E^\ast\oplus S)\to V\oplus V^\ast,\quad [g,e^\ast + s]\mapsto g(x+e^\ast+s).
\]
By \cite[Lemme Fondamental]{LunaSlice} there is a $G$-saturated neighborhood $U^\prime$ of $x\in V\oplus V^\ast$
and an $L$-saturated neighborhood $U$ of $0\in E^\ast\oplus S$ such that $\phi\co G\times_L U\to U^\prime$ is
excellent. By Lemma \ref{lem:ExcellentPullback}, $G\times_L U$ is Hamiltonian with symplectic form
$\phi^*\omega$ and moment mapping $\phi^*\mu$.

\begin{lemma}
\label{lem:SliceMoment}
There are $L$-saturated neighborhoods $Q^\prime$ of $x \in U$ and $Q$ of $0\in E^\ast\oplus S$
and an $L$-equivariant excellent morphism $\widetilde{\psi}\co Q\to Q^\prime$
such that $\widetilde{\psi}^*\phi^*\mu$ has zero set $(\{0\}\times N_S)\cap Q$.
\end{lemma}
\begin{proof}
We make a change of coordinates on $E^*+ S$.
Let $A_1,\dots,A_m$ be elements of $\mathfrak g$ such that the $A_i+\mathfrak l$ form a basis of
$E \simeq \mathfrak{g}/\mathfrak{l}$.
Then the functions $f_i=\mu^{A_i}$ have linearly independent differentials at $x$ and span the dual of $E^\ast$.
Let $h_1,\dots,h_{2k}$ be a basis of $S^\ast$. Then the $f_i$ and $h_j$ are a local coordinate system near
$x\in x+E^\ast+ S$, and the corresponding change of coordinates is equivariant relative to the action of $L$.
By \cite[Lemme Fondamental]{LunaSlice} we have an excellent mapping $\widetilde{\psi}$ of a neighborhood $Q$ of
$0\in E^\ast\times S$ to a neighborhood $Q^\prime$ of $x$ in $x+E^\ast+S$,
which we identify with $E^\ast \oplus S$,
such that the pull-backs of the $f_i$ form
a linear coordinate system on $E^\ast$ and the pull-back of each $h_i$ is $h_i$.
By shrinking if necessary, we may assume $Q^\prime\subset U$.
Evidently, the zero set
of the $f_i$ is $\{0\}\times S$. Let $A\in\mathfrak l$. Then the zero set of $\widetilde{\psi}^*\phi^*\mu^A$ on
$\{0\}\times S$ is just the zero set of $\mu^A$ on $x + S\subset V\oplus V^\ast$. This establishes the lemma.
\end{proof}

Note that $\widetilde{\psi}$ induces an excellent mapping $\psi$ of
$G\times_L Q\to G\times_L Q^\prime \subset G\times_L U$.
Finally, we have the following symplectic slice theorem.

\begin{theorem}
\label{thrm:ShellInSlice}
Let $x$ be a closed orbit in $N$ with isotropy group $L$. Let $(S,L)$ denote the symplectic slice representation at $x$ with its induced Hamiltonian structure as above. Set $E:=T_x(Gx)$. Then there is an $L$-saturated neighborhood $Q$ of $0\in E^\ast\oplus S$ and an excellent morphism $\rho\co G\times_L Q\to V\oplus V^*$
such that $\rho([e,0]) = x$, and if $G\times_LQ$ is given the induced Hamiltonian structure via $\rho$, then the shell of $G\times_LQ$ is
$G\times_L (Q\cap (\{0\}\times N_S))$. The induced mapping
\[
    \widetilde{\rho} \co G\times_L(Q\cap (\{0\}\times N_S))\to N
\]
is excellent.
\end{theorem}
\begin{proof}
It follows from
Lemma \ref{lem:SliceMoment} that $\rho=\phi\circ\psi$ has the desired properties.
In particular, as the shell of $G\times_L Q$ is $G$-invariant, it is determined by its intersection with $Q$,
which is given by $(\{0\}\times N_S)\cap Q$, and $\rho^{-1}(N) = G\times_L(\widetilde{\psi}^*\phi^*\mu^{-1}(0))$.
Since $\rho$ is excellent, $\rho\inv(N)\to N$ is excellent by \cite[Section 6.2]{PopovVinberg},
showing that $\widetilde\rho$ is excellent.
\end{proof}

We now consider the properties of $(V,G)$ that are inherited by $(W,L)$ for a symplectic slice
representation $W\oplus W^\ast$ corresponding to the isotropy group $L$.

\begin{corollary}
\label{cor:ShellSliceProperties}
Let $(P)$ be any of the following properties: reduced,
smooth, or normal. Then $N_S$ has
$(P)$ at $0$ if and only if $N$ has $(P)$ at $x$. Similarly, $N_S\git L$ has (P) at the image of $0$
if and only if $N\git G$ has (P) at the image of $x$.
The complex codimension of $(N\git G)_{(L)}$ in $N\git G$ is the same as the complex codimension of $(N_S\git L)_{(L)}$ in $N_S\git L$, and the complex codimension of $N^{(L)}$ in $N$ is the same as the complex codimension of
$N_S^{(L)}$ in $N_S$.
\end{corollary}
\begin{proof}
Note that the restriction of an excellent mapping to the preimage of a closed invariant set is
excellent, see \cite[Section 6.2]{PopovVinberg}.
For \'etale mappings $\sigma\co X\to Y$, a subvariety $Z$ of $Y$ has property $(P)$ at $y$
if and only if $\sigma\inv(Z)$ has property $(P)$ at a point (equivalently, every point)
of $\sigma\inv(y)$,
see \cite[Corollaire 9.2, Proposition 9.2, and Corollaire 9.10]{SGA1}.
Since $\widetilde\rho$ is \'etale, and since $\widetilde\rho$ induces an \'etale mapping on the quotients,
the claims follow.
\end{proof}

In particular, note that the shell $N$ is a cone and hence connected.
By Corollary \ref{cor:ShellSliceProperties}, if $N$ is normal, then $N_S$ is normal, implying
by Proposition \ref{prop:NormalityShell}({\it ii}) that $(W,L)$ is $1$-modular.

\begin{lemma}
Suppose that $(V,G)$ is $1$-modular with FPIG. Then the principal isotropy groups of $N$ are finite
and $(N\git G)_\pr\subset (N\git G)_\sm$.
\end{lemma}
\begin{proof}
Since $V$ is $1$-modular, $N$ is irreducible by Proposition \ref{prop:NormalityShell}({\it ii}).
As $(V,G)$ has FPIG and $N$ contains $V\times\{0\}$, $N$ has FPIG as well.
Let $Gx\subset N$ be a principal orbit. Since $G_x$  is finite,
Proposition \ref{prop:NormalityShell}({\it iv}) and ({\it v})
show that $x$ is a smooth point of $N$. Hence the slice representation at $x$ is trivial,
and hence $Gx$ is a smooth point of $N\git G$. Thus $(N\git G)_\pr\subset (N\git G)_\sm$.
\end{proof}

Now, suppose that $(V,G)$ is $1$-large, the corresponding shell $N$ is normal,
and $(N\git G)_{\sm} = (N\git G)_{\pr}$, i.e.
every smooth point of $N\git G$ corresponds to a principal orbit. Let $Gx \subset N$
be a principal orbit.
Then the symplectic form $\omega_S$ given by Proposition
\ref{prop:SympSlices} induces a symplectic form on a neighborhood of the orbit $Gx$ in
$N_{\pr}\git G = (N\git G)_{\sm}$.
Recall that $\omega_S$ is given by restriction of the standard symplectic form $\omega$ on $V\oplus V^\ast$.
Hence, $\omega_S$ can be recovered from the Poisson bracket restricted to $S$
via $\omega(X_f,X_g) = \{ f, g\}$ where $X_f$ and $X_g$ denote the Hamiltonian vector fields corresponding to
functions $f$ and $g$. In particular, as the Hamiltonian vector fields of the coordinate functions of $S$ span the
tangent space to $S$, $\omega_S$ is determined by the brackets of the coordinate functions, see Equation
\eqref{eq:PoissonC}. For
$f, g \in \C[V\oplus V^\ast]^G$ representing elements of $\C[N]^G$,
the bracket $\{f + (\mu)^G,g + (\mu)^G\}$ in $\C[N]^G$ is defined by $\{f,g\} + (\mu)^G$, see Definition
\ref{def:CSympQuot}, so that the symplectic form on the image of $S$ in $(N\git G)_{\sm}$ is that induced
by the Poisson bracket on $\C[N]^G$. Hence we have the following.

\begin{corollary}
\label{cor:SympFormSmooth}
Suppose that $(V, G)$ is $1$-large,
the corresponding shell $N$ is normal, and $(N\git G)_{\sm} = (N\git G)_{\pr}$,
and let $\omega$ denote the standard
symplectic form on $V\oplus V^\ast$. Then the Poisson bracket on the complex
symplectic quotient
$\C[N]^G$ induces a regular symplectic form on the smooth locus
$(N\git G)_{\sm}$ of the complex symplectic quotient $N\git G$ that
agrees with the restriction of the standard symplectic form $\omega$ on $V\oplus V^\ast$
to symplectic slices at points with principal isotropy.
\end{corollary}

% xxxxxxxxxxxxxxxxxxxxxxxxxxxxxxxxxxxxxxxxxxxxxxxxxxxxxxxxxxxxxxxxxxxxxxxxx

\subsection{The condition $(N\git G)_{\sm} = (N\git G)_{\pr}$}
\label{subsec:Smooth=Principal}

In this section, we establish conditions under which the hypothesis $(N\git G)_{\sm} = (N\git G)_{\pr}$
of Corollary \ref{cor:SympFormSmooth} is true.

First note that by Theorem \ref{thrm:ShellInSlice}, the hypothesis $(N\git G)_{\sm} = (N\git G)_{\pr}$ is
equivalent to $N_S\git L$ being singular at the origin for every $(W,L)$ corresponding to a symplectic
slice $W\oplus W^\ast$ at a point $x \in N$ with closed orbit and non-principal isotropy group $L$. By
\cite[Lemma 2.3]{HerbigSchwarzSeaton}, $N_S\git L$ is singular if $(W,L)$ is $1$-large so that
$(N\git G)_{\sm} = (N\git G)_{\pr}$ holds whenever each of the $(W,L)$ can be chosen to be $1$-large.

We have the following, which allows us to reduce without loss of generality to the case of
connected groups.

\begin{proposition}
\label{prop:SM=PRConnected}
Let $G$ be reductive and let $V$ be a $G$-module such that the shell $N$ is normal and has FPIG.
If $N\git G$ is smooth, then $N\git G^0$ is smooth.
\end{proposition}
\begin{proof}
Let $x = (v,\xi)\in N$ have image in $(N\git G^0)_{\pr}$. Then $(G^0)_x$ is finite and acts trivially
on the symplectic slice representation $W\oplus W^\ast$ at $x$. Now, $G_x$ also acts on
$W\oplus W^\ast$ and contains no pseudoreflections so that as $(W\oplus W^\ast)/ G_x$ is smooth,
by the Chevalley-Shephard-Todd theorem \cite{ShephardTodd,Chevalley}, $G_x$ must act trivially on $W\oplus W^\ast$.
Therefore, $G/G^0$ has a normal subgroup $H$ which acts
trivially on $(N\git G^0)_{\pr}$, and hence on $N\git G^0$, and the quotient group $F = (G/G^0)/H$ acts
freely on $(N\git G^0)_{\pr}$. We claim that $F$ is trivial.

Let $U$ be the open subset of $\C^{2d} = N\git G$ given by the image of $(N\git G^0)_{\pr}$. Then
$(N\git G^0)_{\pr}$ is a covering space of $U$ with covering group $F$. Because the
strata of $N\git G^0$ have even dimension, the complement of $U$ has complex codimension at least $2$ in $N\git G$.
By \cite[Lemma 2.4]{HerbigSchwarzSeaton}, $U$ is simply connected so that $F$ is trivial.
Hence $N\git G^0$ is smooth.
\end{proof}

We now prove that $(N\git G)_{\sm} = (N\git G)_{\pr}$ when $(V,G)$ is $3$-large,
see Theorem \ref{thrm:Smooth=Princ3Large} below. We begin with the following.

\begin{proposition}
\label{prop:3LargeCodim}
If $(V,G)$ is $3$-large, then for any symplectic slice representation $(S,L)$ at $x\in N$, the codimension
of $N_S\smallsetminus (N_S)_{\pr}$ in $N_S$ is at least $3$.
\end{proposition}
\begin{proof}
As $V$ is $3$-principal, there are $f_1$, $f_2$, $f_3\in\C[V]^G$ which vanish on $V\smallsetminus V_\pr$ and
form a regular sequence in $\C[V]$. As $V$ is 3-large, the $f_i$ and the coordinates $\mu_1,\ldots,\mu_m$
of $\mu$ with respect to a basis for $\mathfrak{g}$ form a regular sequence in $\C[V\oplus V^\ast]$,
see Remark \ref{rem:RegSeqMoment}. Thus the zero set $N^\prime$ of the $f_i$ has codimension $3$ in $N$.
Let $x=(v,\xi) \in N\smallsetminus N^\prime$. Then one of the $f_i(v)$ is not zero, which implies that
$v\in V_\pr$. Hence $x$ is in $N_\pr$. In other words, $N\smallsetminus N_{\pr}\subset N'$ where
$N'$ has codimension $3$ in $N$. By Corollary \ref{cor:ShellSliceProperties}, it follows that
for any symplectic slice representation $(S,L)$, $N_S^{(L)}$ has codimension
at least $3$ in $N_S$.
\end{proof}

We now follow an argument of  \cite{Starr}.
Choose an
$L$-complement $\widehat{W}$ of $W^L$ so that $W = W^L\oplus \widehat{W}$
and let $\widehat{S} = \widehat{W} \oplus \widehat{W}^\ast$ so that $S = S^L\oplus \widehat{S}$.
Then $N_S = S^L\times N_{\widehat{S}}$ and $N_S\git L = S^L\times N_{\widehat{S}}\git L$.
Let $X$ be the image of $N_{\widehat{S}}\smallsetminus\{0\}$ in
$\mathbb{P}:=\mathbb{P}(\widehat{W}\oplus\widehat{W}^\ast)$ and let $Z$ be the image
of the nonzero points in $N_{\widehat{S}}\smallsetminus (N_{\widehat{S}})_{\pr}$. Note that
$X$ is a complete intersection and that $X\smallsetminus Z$ is smooth.

\begin{lemma}
\label{lem:2LargePi12}
The inclusion $X\smallsetminus Z\to X$ induces an isomorphism $\pi_1(X\smallsetminus Z)\to\pi_1(X)$
and a surjection $\pi_2(X\smallsetminus Z)\to\pi_2(X)$.
\end{lemma}
\begin{proof}
Choose a generic linear section $H$ of $\mathbb{P}$  which has codimension $\dim_{\C} X-2$
and does not intersect $Z$. Then \cite[Theorem 1.2]{GM} (with $\hat{n} = 2$) implies that
the maps $\pi_i(H\cap X)\to\pi_i(X)$ and
\[
    \pi_i(H\cap (X\smallsetminus Z)) = \pi_i(H\cap X)\to \pi_i(X\smallsetminus Z).
\]
are isomorphisms for $i=1$ and surjections for $i=2$.
\end{proof}

\begin{theorem}
\label{thrm:Smooth=Princ3Large}
Suppose $(V,G)$ is $3$-large. Then for any symplectic slice representation $(W\oplus W^\ast,L)$
such that $L$ is not principal, $N_{\widehat{S}}\git L$ is not smooth. Hence, $(N\git G)_{\pr} = (N\git G)_{\sm}$.
\end{theorem}
\begin{proof}
Note that if $(V,G)$ is $2$-principal, then the principal isotropy group is the kernel
of the homomorphism $G \to \GL(V)$ by \cite[Corollary 7.7(2)]{GWSlifting}. Hence,
replacing $G$ with its image in $\GL(V)$, we may assume $V$ has TPIG.
As the real shell $M \subset N$ is the Kempf-Ness set and hence intersects each closed orbit in $V$
\cite[Corollary 4.7]{GWSkempfNess}, it follows that $N$ contains closed orbits with trivial isotropy and hence
has TPIG as a $G$-variety.

If $N_{\widehat{S}}\git L$ is smooth, then the same holds for any of the corresponding symplectic slice representations.
Thus it is enough to consider the case that $N_{\widehat{S}}\git L$ is smooth for $(W\oplus W^\ast,L)$
\emph{subprincipal}, i.e., the only proper symplectic slice representations in $W\oplus W^\ast$ are principal.
Since $N_{\widehat{S}}\git L$ is smooth and a cone, it is isomorphic to affine space of dimension $2d$ for some $d\geq 1$.
The non-principal orbits in $N_{\widehat{S}}$ are just the null cone $\NN$ so that the image of the principal orbits in
$N_{\widehat{S}}\git L\simeq \C^{2d}$ is just the complement of the origin.
Therefore, $\pi_i((N_{\widehat{S}}\git L)_{\pr}) = \pi_i(\C^{2d}\smallsetminus\{0\}) = 0$ for $i=1,2$.
Since the action of $L$ on $N_{\widehat{S}}$
has TPIG, we have a fiber bundle $L\to (N_{\widehat{S}})_\pr\to\C^{2d}\smallsetminus \{0\}$. Taking the exact
homotopy sequence, we obtain
\[
 \underbrace{\pi_2(L)}_{=0} \to \pi_2((N_{\widehat{S}})_\pr) \to \underbrace{\pi_2(\C^{2d}\smallsetminus\{0\})}_{=0}
    \to \pi_1(L) \to \pi_1((N_{\widehat{S}})_\pr) \to \underbrace{\pi_1(\C^{2d}\smallsetminus\{0\})}_{=0}.
\]
Thus $\pi_1(L)\simeq \pi_1((N_{\widehat{S}})_\pr)$ and $\pi_2((N_{\widehat{S}})_\pr)=0$.

Recall that $\mathbb{P}=\mathbb{P}^{2\dim_{\C} \widehat{W}-1}$. By \cite[Corollary 9.7]{FL},
the relative homotopy groups $\pi_i(\mathbb{P},X)$ vanish  for
$i\leq 2\dim_{\C} X - \dim_{\C} \mathbb{P} + 1 = 2d$, as $\dim_{\C} N_{\widehat{S}} = 2\dim_{\C} \widehat{W} - \dim_{\C} L$.
It follows that $\pi_2(X)=\Z$ and $\pi_1(X)=0$. We have a diagram of fibrations
\[
\begin{CD}
\C^\times@>>>(N_{\widehat{S}})_\pr@>>>X\smallsetminus Z \\
@VVV @VVV @VVV \\
\C^\times @>>>N_{\widehat{S}}\smallsetminus\{0\}@>>>X.
\end{CD}
\]
A diagram chase and Lemma \ref{lem:2LargePi12} show that
$\pi_1((N_{\widehat{S}})_\pr)\simeq\pi_1(N_{\widehat{S}}\smallsetminus\{0\})$
and that $0=\pi_2((N_{\widehat{S}})_\pr)$ maps onto $\pi_2(N_{\widehat{S}}\smallsetminus\{0\})$, which then must vanish.
From the fibration $\C^\times\to N_{\widehat{S}}\smallsetminus\{0\}\to X$, we then have an exact sequence
\[
    0\to \underbrace{\pi_2(X)}_{=\Z}\to\underbrace{\pi_1(\C^\times)}_{=\Z}
    \to\pi_1(N_{\widehat{S}}\smallsetminus\{0\})\to \underbrace{\pi_1(X)}_{=0}.
\]
Hence $\pi_1(N_{\widehat{S}}\smallsetminus\{0\})$ is finite. Thus $\pi_1((N_{\widehat{S}})_\pr)$
and $\pi_1(L)$ are finite. This implies that $L$ is semisimple.

Now, as $N$ is normal, $(W, L)$ is $1$-modular by Corollary \ref{cor:ShellSliceProperties}.
Then by \cite[Corollaire 1]{LunaVust}, as $W_{(0)}$ is open and consists of points with finite,
hence reductive, isotropy, $(W,L)$ must be stable.
This implies that $(W,L)$ is $1$-large.
By \cite[Lemma 2.3]{HerbigSchwarzSeaton}, it follows that $N_{\widehat{S}}\git L$ is not smooth, a contradiction.
\end{proof}

% xxxxxxxxxxxxxxxxxxxxxxxxxxxxxxxxxxxxxxxxxxxxxxxxxxxxxxxxxxxxxxxxxxxxxxxxx
% xxxxxxxxxxxxxxxxxxxxxxxxxxxxxxxxxxxxxxxxxxxxxxxxxxxxxxxxxxxxxxxxxxxxxxxxx
% xxxxxxxxxxxxxxxxxxxxxxxxxxxxxxxxxxxxxxxxxxxxxxxxxxxxxxxxxxxxxxxxxxxxxxxxx

\section{Proof of Theorem \ref{thrm:GeneralMain}}
\label{sec:GeneralMain}

We are now ready to prove the first of our main results.
We divide the proof of Theorem \ref{thrm:GeneralMain} into two pieces. In Section \ref{subsec:Codim},
we use the results above to establish that $N\git G$ has symplectic
singularities if $(V,G)$ is
3-large, or more generally is $2$-large and
satisfies $(N\git G)_{\pr} = (N\git G)_{\sm}$. In Section \ref{subsec:GradGoren},
assuming that  $(V,G)$ is 1-large, $N$ is normal, and $(N\git G)_{\pr} = (N\git G)_{\sm}$, we show that
$\C[N]^G$ is graded Gorenstein if it is Gorenstein.

% xxxxxxxxxxxxxxxxxxxxxxxxxxxxxxxxxxxxxxxxxxxxxxxxxxxxxxxxxxxxxxxxxxxxxxxxx

\subsection{Symplectic singularities of $N\git G$}
\label{subsec:Codim}

To establish that symplectic quotients that satisfy the hypotheses of Theorem
\ref{thrm:GeneralMain} have symplectic singularities, we first demonstrate the following.

\begin{lemma}\label{lem:FirstCodim}
Let $L$ be a reductive subgroup of $G$ and $Gv$ an orbit in $V$ where $v\in V^L$. Then
the
complex codimension of $(Gv)^L$ in $Gv$ is at least
\[
\dim_\C G-\dim_\C N_G(L)-\dim_\C G_v+\dim_\C L.
\]
\end{lemma}

\begin{proof}
Identify
$Gv$ with $G/G_v$. Then
the tangent space $T_{eG_v} (G/G_v)$ at the coset $eG_v$ of the identity is given by
$\mathfrak{g}/\mathfrak{g}_v$, and because $L$ is reductive,
$T_{eG_v} (G/G_v)^L = \mathfrak{g}^L/(\mathfrak{g}_v)^L$. As $\mathfrak{g}^L = \lie(C_G(L))$,
the $C_G(L)^0$-orbit through $eG_v$ is open in
$(G/G_v)^L$. Equivalently, the orbit of $C_G(L)^0$ in $(Gv)^L$ is open at $v$, hence open
at any point in $(Gv)^L$. It follows that $C_G(L)^0$ acts transitively on each connected
component of $(Gv)^L$. Hence the complex codimension of $(Gv)^L$ in $Gv$ is
\begin{align*}
    \dim_\C G - &\dim_\C G_v - \big(\dim_\C C_G(L) - \dim_\C C_{G_v}(L)\big)
    \\&\geq     \dim_\C G - \dim_\C G_v-\dim_\C C_G(L)   + \dim_\C C(L)  -\dim_\C L+\dim_\C L
    \\&= \dim_\C G-\dim_\C G_v-\dim_\C N_G(L)+\dim_\C L
\end{align*}
where we have used the fact that   $N_G(L) /L= (G/L)^L$ has complex dimension $\dim_\C C_G(L)  - \dim_\C C(L).$
\end{proof}

\begin{lemma}
\label{lem:Codim}
Suppose that $k \geq 1$ and that $(V,G)$ is $k$-modular, and let $L$ be a reductive subgroup of $G$.
Then for each $p > 0$,
\[
    \codim_{\C} (V^L\cap V_{(p)}) \geq \dim_\C G - \dim_\C N_G(L) + \dim_\C L + k.
\]
\end{lemma}
\begin{proof}
Let $v\in V^L$  and let $p = \dim_\C G_v > 0$. Then as $(V, G)$ is $k$-modular,
$\codim_\C V_{(p)} \geq p + k$. Hence by Lemma \ref{lem:FirstCodim} we see that
the complex codimension of $V^L\cap V_{(p)}$ in $V$ is at least
\[
    p + k+ \dim_\C G - p-\dim_\C N_G(L) + \dim_\C L
    =
    \dim_\C G - \dim_\C N_G(L) + \dim_\C L + k.
    \qedhere
\]
\end{proof}

\begin{lemma}\label{lem:CodimFinite}
Suppose that $k \geq 1$,   $(V,G)$ is   $k$-large  with TPIG, and $L$ is  a nontrivial finite subgroup of $G$. Then
\[
    \codim_{\C} V^L \geq\dim_\C G-\dim_\C N_G(L)+k.
\]
\end{lemma}

\begin{proof}
For each $p > 0$, $\codim_{\C} (V^L\cap V_{(p)}) \geq \dim_\C G - \dim_\C N_G(L) + k$ by Lemma \ref{lem:Codim}.
If $v \in V^L\cap V_{(0)}$, then $Gv$ lies in $V\smallsetminus V_{\pr}$ which has complex codimension at least $k$ in $V$.
Applying Lemma \ref{lem:FirstCodim} completes the proof.
\end{proof}
We now can prove the following.

\begin{theorem}
\label{thrm:Codim}
Let $k \geq 1$ and let $(V,K)$ be unitary where $(V,G)$ is $k$-large,
the corresponding shell $N$ is normal, and $(N\git G)_{\pr} = (N\git G)_{\sm}$.
Then $N\git G\smallsetminus (N\git G)_{\sm}$ has complex codimension at least $2k$ in $N\git G$.
\end{theorem}
\begin{proof}
Since the isotropy type strata of $N\git G$ are even dimensional, the case $k=1$ is trivial and we may assume that $k\geq 2$. Then as in the proof of Theorem \ref{thrm:Smooth=Princ3Large}, we may assume that $N$ has TPIG as a $G$-variety.

Let $n=\dim_\C V$. Let $L$ be a non-principal isotropy group for the action of $G$ on $N$ and let $(v,\xi)\in N$ have closed orbit and isotropy group $L$.
Let $S = W\oplus W^\ast$ be the corresponding  symplectic slice. By Corollary
\ref{cor:ShellSliceProperties},  the complex
codimension of $(N\git G)_{(L)}$ in $N\git G$ is the same as the complex codimension of $(N_S\git L)_{(L)}$ in $N_S\git L$.
As $N$ has TPIG as a $G$-variety so that $N_S$ has TPIG as an $L$-variety,
$\dim_\C N_S\git L = \dim_\C S - 2\dim_\C L$.

As an $L$-module,
\[
    V\oplus V^\ast = S\oplus \mathfrak{g}/\mathfrak{l} \oplus (\mathfrak{g}/\mathfrak{l})^\ast.
\]
Write $S = S^L\oplus \widehat{S}$. Then the complex codimension of $(N_S\git L)_{(L)}$ in $N_S\git L$ is
$\dim_\C \widehat{S} - 2\dim_\C L$. From the decomposition above,
\[
    \dim_\C S^L = \dim_\C (V\oplus V^\ast)^L - 2\dim_\C\mathfrak{g}^L + 2\dim_\C\mathfrak{l}^L,
\]
so that as $\dim_\C S = 2n - 2(\dim_\C G - \dim_\C L)$, we have
\[
    \dim_\C \widehat{S}
        = 2n - \dim_\C(V\oplus V^\ast)^L - 2\dim_\C G + 2\big(\dim_\C C_G(L) + \dim_\C L - \dim_\C C(L)\big).
\]
Using the fact that $\dim_\C N_G(L) = \dim_\C C_G(L) + \dim_\C L - \dim_\C C(L)$, see the proof of Lemma
\ref{lem:FirstCodim}, this gives
\[
    \dim_\C \widehat{S} = 2n - \dim_\C(V\oplus V^\ast)^L - 2\dim_\C G + 2\dim_\C N_G(L).
\]
Then as $\dim_\C(V\oplus V^\ast)^L = 2\dim_\C V^L$,
applying Lemmas \ref{lem:Codim} and \ref{lem:CodimFinite} yields $\dim_\C \widehat{S} \geq 2\dim_\C L + 2k$, and
the complex codimension of $(N_S\git L)_{(L)}$ in $N_S\git L$ is at least $2k$.
\end{proof}

Now, suppose that $(V,G)$ is $2$-large and $(N\git G)_{\pr} = (N\git G)_{\sm}$;
by Theorem \ref{thrm:Smooth=Princ3Large}, this is the case if $(V,G)$ is $3$-large.
As $(V,G)$ is $2$-modular, $N$ is normal,
implying that $N\git G$ is normal. The smooth points of $N\git G$ inherit a
symplectic $2$-form from that of $V$ by Corollary \ref{cor:SympFormSmooth}.
Then combining Theorem \ref{thrm:Codim} and Theorem \ref{thrm:FlennerCodim},
we have the following.

\begin{corollary}
\label{cor:SymplecticSing}
Suppose that $(V,G)$ is $2$-large and $(N\git G)_{\pr} = (N\git G)_{\sm}$.
Then the complex symplectic quotient $N\git G$ has symplectic singularities.
\end{corollary}

That $N\git G$ is
Gorenstein with rational singularities then follows from Theorem \ref{thrm:NamikawaRatGoren}.
To complete the proof of Theorem \ref{thrm:GeneralMain}, then, we need only
demonstrate that $N\git G$ is graded Gorenstein; this will be shown in the next section.

% xxxxxxxxxxxxxxxxxxxxxxxxxxxxxxxxxxxxxxxxxxxxxxxxxxxxxxxxxxxxxxxxxxxxxxxxx

\subsection{Graded Gorenstein symplectic quotients}
\label{subsec:GradGoren}

The goal of this section is to prove the following.

\begin{theorem}
\label{thrm:GradeGoren}
Assume that $(V,G)$ is $1$-large, the corresponding shell $N$ is normal, and $(N\git G)_{\pr} = (N\git G)_{\sm}$.
If $N\git G$ is Gorenstein, then it is graded Gorenstein.
\end{theorem}

Clearly, Theorem \ref{thrm:GradeGoren}
and Corollary \ref{cor:SymplecticSing} imply Theorem \ref{thrm:GeneralMain}.
However, we will also have occasion to apply Theorem \ref{thrm:GradeGoren}
in cases that are not $3$-large; see Section \ref{sec:SU2}.

Let $R$ be a finitely generated $\N$-graded algebra over $\C$ of dimension $d$ with $R_0=\C$.  Assume that $R$ is Gorenstein
and normal. Let $K_R$ denote the canonical module of $R$ so that
$K_R$ is graded isomorphic to $R$ with a degree shift; see Section \ref{subsec:BackSympSing}. Let $\sigma$ be a
homogeneous module generator of $K_R$, and let $X := \Spec R$.

Let $I$ be the ideal of functions on $X\times X$ of the form $\sum_i f_i(x)\otimes h_i(y)$
such that $\sum_i f_i(x) h_i(x) = 0$, and let $\Omega(X) = I/I^2$. Then $I$ is a homogeneous ideal,
and the mapping $f\mapsto df \in \Omega(X)$ defined by $f \mapsto f\otimes 1 - 1 \otimes f + I^2$
is degree-preserving. Note that $\Omega(X)$ is a graded $R$ module by multiplication on the first
factor. Then the map $(h_1,\dots,h_d)\mapsto dh_1\wedge\dots\wedge dh_d$ is degree-preserving and maps
$d$-tuples of homogeneous functions to homogeneous elements of $\wedge^d(\Omega(X))$. Then as $\sigma$
is a homogenous generator of $K_R$, it is homogeneous when considered as a section of $\wedge^d(\Omega(X))$.
On the smooth locus $X_\sm$ of $X$, we can identify sections of $K_R$ with sections of $\wedge^d(\Omega(X))$. Since $X$ is normal,
$\C[X_{\sm}] = \C[X]$.  See \cite[Section II.8]{Hartshorne} or
\cite[Section 8.0]{CoxLittleSchenckToric}.

Let $f_1,\dots,f_d$ be a homogeneous regular sequence for $R$ and let $d_i := \deg f_i$.
On $X_\sm$, we have $df_1\wedge\dots\wedge df_d=\lambda\sigma$ where $\lambda$ is a regular function
on $X_{\sm}$. Since $X$ is normal, $\lambda$ extends to a regular function on $X$, and as $\sigma$ and
the $f_i$ are homogeneous, $\lambda$ is as well. Let $m:=\deg \lambda$. Note that if we calculate
$\lambda$ in $K_R$, whose grading may be different than that on $\wedge^d\Omega(X)$, the ratio
$\lambda$ remains the same.

Let $S=\C[f_1,\dots,f_d]$. Since $R$ is Gorenstein, $R=S\otimes H$ where $H$ is a finite dimensional
complex vector space whose Hilbert series has the form $1+\dots+t^r$. As a graded module,
$K_R=f_1\cdots f_d S\otimes H^*$ where $H^*$ has Hilbert series $1+\dots+t^{-r}$.
We demonstrate this with the following.

\begin{lemma}
\label{lem:DegreeCanonMod}
The image of $df_1\wedge\dots\wedge df_d$ in $K_R$ is $f_1\cdots f_d\otimes 1$ and hence has degree $\sum d_i$.
\end{lemma}
\begin{proof}
By the definition of $K_R$ we have that $K_R^*$ is calculated as the $n$th cohomology of the complex
\begin{align*}
    0\to R\to\oplus_i R[f_i\inv]\to\dots&\to  R[f_2\inv,\dots,f_d\inv]\oplus\dots\oplus R[f_1\inv,\dots,f_{d-1}\inv]
    \\
    &\to R[f_1\inv,\dots,f_d\inv]\to 0,
\end{align*}
where all the mappings are inclusions up to a sign; see \cite[Section 2.1]{GotoWatanabe} and
\cite[Section 2.1]{HunekeLocalCohom}. Because $R=S\otimes H$, we may use the corresponding complex for $S$
and tensor each term with $H$. Hence, it is enough to consider the case where $R$ is replaced with $S$.

First consider the case of a polynomial ring in one variable $x$ of degree $e$. In this case, we have the complex
\[
    0   \to   \C[x]   \to \C[x,x\inv]   \to 0
\]
with cokernel $x\inv\C[x\inv]$. Taking the dual, we see that  $K_{\C[x]}\simeq x\C[x]$ with the canonical grading.
The mapping of $\Omega(\Spec \C[x])$ to $K_{\C[x]}$ sends the generator $dx$ of $\Omega(\Spec \C[x])$ to the
generator $x$ of $K_{\C[x]}$ and hence preserves degrees. By an induction argument, it follows that
in our case, $K_R\simeq f_1\cdots f_d\cdot S\otimes H^*$ with its canonical grading,
and the image of $df_1\wedge\dots\wedge df_d$ is $f_1\cdots f_d\otimes 1$.
\end{proof}

\begin{corollary}
\label{cor:DegreePreserv}
We have $r=m$ so that the isomorphism between sections of $K_R$ and $\wedge^d\Omega(X)$ is degree preserving.
In particular, $\sigma$ has degree $\sum d_i-m$.
\end{corollary}
\begin{proof}
The degree of $\sigma$ in $\wedge^d\Omega(X)$ is $\sum d_i-m$ and its degree in $K_R$ is $\sum d_i-r$.
Hence $r=m$.
\end{proof}

We are now ready to prove the main result of this section.

\begin{proof}[Proof of Theorem \ref{thrm:GradeGoren}]
Let $(V,G)$ be $1$-large and assume that the shell $N$ is normal and that
$\C[V\oplus V^\ast]^G/(\mu)^G$
is Gorenstein. Let $n = \dim_\C V$ and $g = \dim_\C G$. Let $\vol$ be the volume element on $V\oplus V^\ast$
so that, up to sign, $\vol = dx_1\wedge\dots\wedge dx_n\wedge dy_1\wedge\dots\wedge dy_n$ has degree
$2n$ and is $G$-invariant. Choosing a basis $A_1,\ldots,A_g$ for $\mathfrak{g}$, we let $\mu_1,\dots,\mu_g$ denote the corresponding
entries of the moment map with respect to the dual basis for $\mathfrak{g}^\ast$.

Note that as $(V,G)$ is $0$-modular, the shell $N$ is a complete intersection by
\cite[Proposition 9.4]{GWSlifting}.
In particular, the algebra of regular functions on the shell
$\C[N] = \C[V\oplus V^\ast]/(\mu)$
is given by the quotient of the polynomial algebra $\C[V\oplus V^\ast]$ in $2n$ variables by the complete
intersection $(\mu)$ generated by $g$ quadratics, and hence
has Hilbert series
\[
    \Hilb_{\C[N]}(t)  =   \frac{(1 - t^2)^g}{(1 - t)^{2n}}
\]
so that by Theorem \ref{thrm:StanleyGorenHilb}, the $a$-invariant of
$\C[N]$ is
$2g-2n$. Hence, the generator $\sigma_N$ of the canonical module of $\C[N]$ has degree
$2n - 2g$. To describe $\sigma_N$ explicitly, if $f_1,\ldots,f_{2n-g} \in \C[V\oplus V^\ast]$
are homogeneous, representing homogeneous elements of $\C[N]$, define
\[
    \alpha_N(f_1,\dots,f_{2n-g})
    =   df_1\wedge\dots\wedge df_{2n-g}\wedge d\mu_1\wedge\dots\wedge d\mu_g/\vol,
\]
restricted to $N$. Note that $\alpha_N$ depends only on the elements of $\C[N]$ represented
by the $f_i$. Because the $d\mu_i$ have degree $2$, the value of $\alpha_N(f_1,\dots,f_{2n-g})$
is homogeneous of degree $\sum_i\deg f_i-(2n-2g)$.  Then the form $\sigma_N$ is defined by
\[
    (df_1\wedge\dots\wedge df_{2n-g})|_N/\sigma_N=\alpha_N(f_1,\dots,f_{2n-g}).
\]
Note that  $h\in G$ acts on $\sigma_N$   by the determinant of $h$ acting on $\mathfrak{g}^*$. Moreover, by Proposition
\ref{prop:NormalityShell}({\it iv}), $\sigma_N$ does not vanish on the smooth locus of $N$.

Define
\[
    \sigma_{N\git G}=\iota_{A_1}\dots \iota_{A_g}\sigma_N
\]
where $\iota_{A_i}$ denotes  contraction and the basis elements $A_i$ for $\mathfrak{g}$ are
identified with the corresponding vector fields on $N$. Then $\sigma_{N\git G}$ is $G$-invariant
and gives rise to an element of $\wedge^{2n-2g}\Omega((N\git G)_\sm)$.
Because
$(N\git G)_{\pr} = (N\git G)_{\sm}$, we have by
Proposition \ref{prop:NormalityShell}({\it iv}) and ({\it v}) that the image under the orbit map
$\pi\co N\to N\git G$ of the smooth locus $N_{\sm}$ contains $(N\git G)_\sm$. Recalling that $N$ and $N\git G$
are normal, as $\sigma_N$ does not vanish on $N_{\sm}$, it follows that $\sigma_{N\git G}$ is a generator of $\wedge^{2n-2g}\Omega(N\git G_\sm)$, and hence of the canonical module of
$\C[N]^G$.
Because the vector fields $A_i$ have degree zero, the degree of $\sigma_{N\git G}$ is $2n-2g$. Noting that
$\dim_\C N\git G = 2n - 2g$ completes the proof.
\end{proof}

% xxxxxxxxxxxxxxxxxxxxxxxxxxxxxxxxxxxxxxxxxxxxxxxxxxxxxxxxxxxxxxxxxxxxxxxxx
% xxxxxxxxxxxxxxxxxxxxxxxxxxxxxxxxxxxxxxxxxxxxxxxxxxxxxxxxxxxxxxxxxxxxxxxxx
% xxxxxxxxxxxxxxxxxxxxxxxxxxxxxxxxxxxxxxxxxxxxxxxxxxxxxxxxxxxxxxxxxxxxxxxxx

\section{The case of $G^0$ a torus}
\label{sec:Torus}

In this section, we prove Theorem \ref{thrm:TorusMain}. Because the shell $N$ depends only on
$G^0$, we first consider the case that $G = G^0$ is a torus.
Assume that $G = (\C^\times)^\ell$ for some positive integer $\ell$. Without loss of generality,
assume that $V^G = \{0\}$ for simplicity, and by replacing $G$ with its image in $\GL(V)$,
we assume that $(V,G)$ is faithful.

Choosing a basis for $V$ with respect to which the action of
$G$ is diagonal, we may identify $V$ with $\C^n$ and describe the action of $G$ with
a weight matrix $A = (a_{ij}) \in \Z^{\ell\times n}$. Specifically, in coordinates
$(z_1,\ldots,z_n)$ with respect to our basis, the action of $G$ is given by
\[
    (t_1,\ldots,t_\ell)\cdot(z_1,\ldots,z_n)
    :=
    \left(
        z_1\prod\limits_{i=1}^\ell t_i^{a_{i1}}, \ldots,
        z_n\prod\limits_{i=1}^\ell t_i^{a_{in}}
    \right).
\]
for $(t_1,\ldots,t_\ell)\in G = (\C^\times)^\ell$. Choosing the dual basis for $V^\ast$
and using coordinates
\linebreak
$(z_1,\ldots,z_n,w_1,\ldots,w_n)$ with respect to the concatenated
basis for $V\oplus V^\ast$, the weight matrix for the action of $G$ on $V\oplus V^\ast$
is $(A|-A)$. Using the symplectic form
$\omega = \sqrt{-1}/2 \sum_{j=1}^n dz_j \wedge d\overline{z}_j$ on $V$, and identifying
$\mathfrak{k}^\ast$  with $\R^\ell$ via the coordinates for $(\C^\times)^\ell$ used above,
the real moment map $J\co V\to \mathfrak{k}^\ast$ is given by the component functions
\begin{equation}
\label{eq:TorusMomentR}
    J_i(\bs{z},\bs{\overline{z}}) = \sum\limits_{j=1}^n a_{ij} z_j \overline{z_j},
    \quad\quad
    i=1,\ldots,\ell.
\end{equation}
It follows that the complexified moment map $\mu\co V\oplus V^\ast \to \mathfrak{g}^\ast$
is given by the component functions
\begin{equation}
\label{eq:TorusMoment}
    \mu_i(\bs{z},\bs{w}) = \sum\limits_{j=1}^n a_{ij} z_j w_j,
    \quad\quad
    i=1,\ldots,\ell.
\end{equation}
See \cite{FarHerSea,HerbigIyengarPflaum,HerbigSeaton2} for more details.

Using \cite[Lemma 2]{WehlauPopov}, we may reduce to the case that $(V,G)$ is stable;
see also \cite[Section 3]{FarHerSea} and \cite[Lemma 3]{HerbigSeaton}.
Specifically, we have the following.

\begin{lemma}
\label{lem:TorusStableRed}
Let $G = (\C^\times)^\ell$ and let $V$ be a faithful $G$-module with $V^G = \{0\}$.
Then there is a submodule $V^\prime$ of $V$ such that $(V^\prime, G^\prime)$ is stable and faithful,
where $G^\prime$ is the quotient of $G$ by the subgroup that acts trivially on $V^\prime$.
Moreover, the real shell and the real and complex symplectic quotients corresponding to $(V,G)$ and
$(V^\prime, G^\prime)$ coincide.
\end{lemma}
\begin{proof}
If $(V,G)$ is stable, then there is nothing to prove, so assume not.
Then $V\git G$ has complex dimension strictly less than $n - \ell$. By
the Kempf-Ness homeomorphism \cite[Corollary 4.7]{GWSkempfNess}, the real symplectic quotient
$M_0 = M/K$ has
real dimension less than $2n - 2\ell$ from which it follows that the real shell
$M$ has real dimension less than $2n - \ell$. Considering the moment map in Equation \eqref{eq:TorusMomentR},
it follows that there are coordinates $z_i$ that vanish on $M$; again by the Kempf-Ness homeomorphism,
these coordinates
do not appear in any invariant polynomial.
By \cite[Lemma 2]{WehlauPopov}, restricting to
the subspace of $V^\prime$ of $V$ on which all such coordinates vanish yields a stable torus-module
$(V^\prime,G^\prime)$ such that $V\git G = V^\prime\git G^\prime$. Moreover, as the real shell $M$ is contained
in $V^\prime$,
restricting to $(V^\prime,G^\prime)$ does not change the real symplectic
quotient $M_0$. As the complex symplectic quotient
is defined in Definition \ref{def:CSympQuot} to be
the complexification of the real symplectic
quotient, it also does not change the complex symplectic quotient.
\end{proof}

Note that Lemma \ref{lem:TorusStableRed} can be understood as an application of the Luna-Richardson
Theorem, \cite[Corollary 4.4]{LunaRichardson}. Specifically, if $H$ is the principal isotropy group
of $V$, then inclusion yields an isomorphism $\C[V^H]^{N_G(H)/H} = \C[V^H]^{G/H} \simeq \C[V]^G$.
Because $H$ acts with no closed orbits on a $G$-complement to $V^H$, the invariants only involve
coordinates of $V^H$. By construction, $(V^H, G/H)$ has FPIG and hence is stable, and the real shell $M$ is contained in $V^H$.

Reducing to the case of $(V,G)$ stable and faithful, we can now demonstrate the following.
See also \cite[Theorem 2.2]{Bulois}, where it was demonstrated independently and simultaneously without the hypothesis of stable
that $N$ is normal if and only if it is irreducible.

\begin{lemma}
\label{lem:TorusNormalShell}
Suppose that $G = (\C^\times)^\ell$ is a torus and that $(V,G)$ is stable and faithful. Then both the shell
$N$ and the complex symplectic quotient $N\git G$ are normal varieties.
\end{lemma}
\begin{proof}
By \cite[Theorem 3.2]{HerbigSchwarz}, \cite[Proposition 1]{VinbergComplexity},
$(V,G)$ is $1$-large, and by Lemma \ref{lem:dimVminusdimG}, it follows that
$\dim_\C V > \dim_\C G$. The shell $N$ is a complete intersection and Cohen-Macaulay
by Proposition \ref{prop:NormalityShell}({\it i}). By \cite[Corollary 4.3]{HerbigSchwarz},
$(\mu)$ is radical and hence the ideal of $N$.
Because row-reducing the weight matrix $A$ over $\Z$ corresponds to choosing a
different basis for $\mathfrak{g}^\ast$ and corresponding coordinates for $G$, we may assume that
the weight matrix is in the form $A = (D|C)$ where $D$ is diagonal and
has strictly negative entries.

Using Equation \eqref{eq:TorusMoment}, we express the Jacobian matrix of the moment map
$\mu\co V\oplus V^\ast \to \mathfrak{g}^\ast$ as
\[
    d\mu = (D_1 | \cdots | D_2 | \cdots )
\]
where $D_1$ is the $\ell\times\ell$ diagonal matrix with $(i,i)$-entry $a_{ii}w_i$,
$D_2$ is the $\ell\times\ell$ diagonal matrix with $(i,i)$-entry $a_{ii}z_i$,
and the $\cdots$ indicate $\ell\times(n-\ell)$ blocks.
As $a_{ii}\neq 0$ for $i \leq \ell$, it follows from Proposition \ref{prop:NormalityShell}({\it iv})
that $(\bs{z},\bs{w})\in V\oplus V^\ast$ is a smooth point of $N$ unless $z_j = w_j = 0$ for some $j$.
Therefore, $N$ is smooth in codimension $1$.
Then by Serre's Criterion \cite[Theorem 23.8]{MatsumuraBook}, $N$ is Cohen-Macaulay
and regular in codimension $1$, hence normal.
\end{proof}

\begin{remark}
Note that there is a mistake in the proof of \cite[Theorem 3]{HerbigSeaton2}, where it is incorrectly
assumed that the torus representation is $2$-large. Lemma \ref{lem:TorusNormalShell} offers an
alternate proof that corrects this mistake.
\end{remark}

We now have the following.

\begin{proposition}
\label{prop:TorusRationalShell}
Suppose that $G = (\C^\times)^\ell$ is a torus and that $(V,G)$ is stable and faithful. Then both the shell
$N$ and the complex symplectic quotient $N\git G$ have rational singularities.
\end{proposition}
\begin{proof}
If the symplectic quotient is a point, then the result is trivial, so assume not.
As above, we assume without loss of generality that $V^G=\{0\}$ and choose coordinates with respect
to which $A = (D|C) \in \Z^{\ell\times n}$, $\ell< n$, with $D$ diagonal having strictly negative entries
on the diagonal. Recalling that $N$ is a complete intersection
with $(\mu)$ generated
by $\ell$ quadratics, the Hilbert series of
$\C[N] = \C[V\oplus V^\ast]/(\mu)$ is given by
\begin{equation}
\label{eq:TorusShellHilbSer}
    \Hilb_{\C[N]}(t) = \frac{(1 - t^2)^\ell}{(1 - t)^{2n}}
\end{equation}
so that the $a$-invariant of $N$ is $2(\ell-n)$. In particular, $\ell < n$ implies that the
$a$-invariant is negative. With this, the proof is by induction on $\ell$.

If $\ell = 1$, then the moment map $\mu\co V\oplus V^\ast \to \mathfrak{g}^\ast$ is given by
the single function
\[
    \mu(\bs{z},\bs{w}) = \sum\limits_{j=1}^n a_j z_j w_j
\]
with Jacobian
\[
    d\mu(\bs{z},\bs{w}) = (a_1 w_1, \ldots, a_n w_n, a_1 z_1, \ldots, a_n w_n).
\]
As $V^G = \{0\}$, each $a_j \neq 0$ so that $d\mu$ has rank $1$ away from the origin. It follows that
$N$ is Cohen-Macaulay, normal, has negative $a$-invariant, and is regular away from the origin so that,
by a theorem of Flenner and Watanabe
\cite[Satz 3.1]{FlennerRational}, \cite[Theorem 2.2]{WatanabeRational},
\cite[Theorem 9.2]{HunekeTight}, $N$ has rational singularities.

Now assume that $G$ has
complex dimension $\ell$ and that the result holds for representations of tori of dimension
less than $\ell$. We have established that
$\C[N]$ is normal, Cohen-Macaulay, and has negative $a$-invariant,
so again by the theorem of Flenner and Watanabe, it suffices to show that $N$ has rational singularities away from the
origin. Let $p \in N$ be a nonzero point, and then $p$ has a nonzero coordinate. By switching $V$ and $V^\ast$
and permuting the bases for $V$ and $V^\ast$, we may assume that $p$ has nonzero $z_1$-coordinate. Recalling that
$V^G = \{0\}$ and row-reducing $A$ (i.e., changing bases for $\mathfrak{g}^\ast$) results in a weight matrix
$A = (D|C)$ such that $a_{11}\neq 0$ and $a_{1i} = 0$ for $i = 2,\ldots,\ell$. Then
\[
    \mu_1(\bs{z},\bs{w})  = a_{11}z_1 w_1 + \sum\limits_{j=\ell+1}^n a_{1j} z_j w_j,
\]
and each $\mu_i$ for $i > 1$   involves neither $z_1$ nor $w_1$. Localizing to the set $z_1 \neq 0$,
we adjoin an inverse to $z_1$ and express the condition $\mu_1 = 0$ as
\[
    w_1 = - \sum\limits_{j=\ell+1}^n \frac{a_{1j} z_j w_j}{a_{11}z_1},
\]
eliminating $w_1$. Then on the set $z_1\neq 0$, $N$ is
isomorphic to the product of $\C$
with the shell associated to the $(\ell-1)$-dimensional torus action with weight matrix formed by
removing the first row and column from $A$. This has rational singularities by the inductive hypothesis,
completing the proof
that $N$ has rational singularities. Then $N\git G$ has rational singularities by a theorem of Boutot
\cite{Boutot}.
\end{proof}

As described above, if $(V,G)$ is not stable, then the corresponding complex symplectic quotient
coincides with that of a stable torus representation so that the symplectic quotient still has
rational singularities. However, the shell $N$ may fail to
have rational singularities.

\begin{example}
\label{ex:TorusShellNotRational}
Let $G = (\C^\times)^n$ act on $V = \C^n$ with weight matrix given by the identity matrix.
Then the only closed orbit is the origin so that $(V,G)$ is not stable. Noting that
$\ell = n$ in Equation \eqref{eq:TorusShellHilbSer}, one observes that the $a$-invariant
of the shell $N$ is zero so that, again by the theorem of Flenner and Watanabe
\cite[Satz 3.1]{FlennerRational}, \cite[Theorem 2.2]{WatanabeRational}, \cite[Theorem 9.2]{HunekeTight},
$N$ does not have rational singularities.
\end{example}

In order to apply Theorem \ref{thrm:GradeGoren} to this case, we will need the following.

\begin{lemma}
\label{lem:TorusSingular}
Let $G = (\C^\times)^\ell$ be a torus and let $(V,G)$ be $1$-modular. Then $N\git G$ is not smooth.
\end{lemma}
\begin{proof}
As above, assume without loss of generality that $V^G = \{0\}$ and $V$ is faithful,
choose coordinates $z_1,\ldots,z_n$ on $V\simeq \C^n$ with respect to which the action of $G$
is diagonal, and let $w_1,\ldots,w_n$ be the dual coordinates. Because $V$ is $1$-modular,
$\ell \leq n - 1$ by Lemma \ref{lem:dimVminusdimG}.

Assume for simplicity that the characters corresponding to $z_1,\ldots,z_\ell$ are linearly independent
in the character group $X(G)$ of $G$. Then there are $n - \ell$ indices $i$ such that $z_i w_i$
does not vanish on $N$. For each index $k > \ell$, there is a minimal invariant
of the form $z_k^{a_k} w_k^{b_k}\prod_{i=1}^\ell z_i^{a_i} \prod_{i=1}^\ell w_j^{b_j}$ where the $a_i$
and $b_j$ are nonnegative and $a_i b_i = 0$ for all $i$. Dually, there is a minimal invariant
$w_k^{a_k} z_k^{b_k}\prod_{i=1}^\ell w_i^{a_i} \prod_{i=1}^\ell z_j^{b_j}$.
Hence, there are at least $3(n - \ell)$ minimal generators, while $\dim_{\C} N\git G \leq 2(n - \ell)$,
so that $N\git G$ is not smooth.
\end{proof}

If $(V,G)$ is $1$-large and $G^0$ is a torus, then for any symplectic slice representation
$(W\oplus W^\ast,L)$, $(W,L^0)$ is $1$-modular by Corollary \ref{cor:ShellSliceProperties}.
Hence, applying Lemma \ref{lem:TorusSingular} to the local model $N_S\git L$ of $N\git G$
corresponding to each $(W,L^0)$ as well as
Proposition \ref{prop:SM=PRConnected}, we have that $(N\git G)_{\pr} = (N\git G)_{\sm}$
in this case. Note that if $(V,G)$ is not $1$-large, then $N\git G$ may not be the complex
symplectic quotient of Definition \ref{def:CSympQuot}.

\begin{proof}[Proof of Theorem \ref{thrm:TorusMain}]
Assume that $G$ is a reductive group such that $G^0$ is a torus. Let $V$ be a $G$-module, and let
$V = V_1\oplus V_2$ where $V_2 = V^{G^0}$. Then letting $M$ denote the real shell associated to $(V_1,G^0)$
the real symplectic quotient associated to $(V,G)$ is given by $(M\times V_2)/K = ((M/K^0)\times V_2)/(K/K^0)$.
By Lemma \ref{lem:TorusStableRed}, we may replace $(V_1,G^0)$
with a faithful stable torus module $(V_1^\prime, (G^0)^\prime)$ such that the real shell $M$
and the real and complex symplectic quotients of $(V_1,G^0)$ and $(V_1^\prime, (G^0)^\prime)$
coincide. Let $N$ denote the complex shell associated to
$(V_1^\prime, (G^0)^\prime)$ and then the corresponding complex symplectic quotient is given by
$N\git (G^0)^\prime$. By Proposition \ref{prop:TorusRationalShell}, both $N$ and $N\git (G^0)^\prime$
have rational singularities.
Similarly, Boutot's theorem \cite{Boutot} implies that $N\git G$ has rational singularities.

Recall that the Hilbert series of regular functions on the real and
complex symplectic quotients coincide. The complex symplectic quotient $N\git (G^0)^\prime$ is then
graded Gorenstein by \cite[Theorem 1.3 and Corollary 1.8]{HerbigHerdenSeaton}. Then by
Theorem \ref{thrm:NamikawaRatGoren}, $N\git (G^0)^\prime$ has symplectic singularities.
In the case $G = G^0$, it follows that the complex symplectic quotient $N\git (G^0)^\prime$ associated
to $(V,G)$ is graded Gorenstein with symplectic singularities, completing the proof in this case.
Otherwise, $(N\git (G^0)^\prime) \times (V_2\oplus V_2^\ast)$ has symplectic singularities so that by
\cite[Proposition 2.4]{Beauville}, the symplectic quotient
$\big((N\git (G^0)^\prime) \times (V_2\oplus V_2^\ast)\big)/(G/G^0)$
associated to $(V,G)$ has symplectic singularities.
Finally, let $\sigma_{(N \times V_2\oplus V_2^\ast)\git (G^0)^\prime}$ denote the generator of the
canonical module of $\C[N\times V_2\oplus V_2^\ast]^{(G^0)^\prime}$ constructed in
the proof of Theorem \ref{thrm:GradeGoren}. An element $g\in G$ acts on
$\sigma_{N \times V_2\oplus V_2^\ast}$ by the determinant of the action on $\mathfrak{g}^\ast$
and acts on $A_1\wedge\cdots\wedge A_\ell$ by the determinant of the action on $\mathfrak{g}$;
as these determinants are inverses, $\sigma_{(N \times V_2\oplus V_2^\ast)\git (G^0)^\prime}$
is in fact $G$-invariant. Hence $\C[N\times V_2\oplus V_2^\ast]^{(G^0)^\prime}$ and
$\C[N\times V_2\oplus V_2^\ast]^G$ have the same $a$-invariant and dimension, implying that
$\C[N\times V_2\oplus V_2^\ast]^G$ is graded Gorenstein and completing the proof.
\end{proof}

% xxxxxxxxxxxxxxxxxxxxxxxxxxxxxxxxxxxxxxxxxxxxxxxxxxxxxxxxxxxxxxxxxxxxxxxxx
% xxxxxxxxxxxxxxxxxxxxxxxxxxxxxxxxxxxxxxxxxxxxxxxxxxxxxxxxxxxxxxxxxxxxxxxxx
% xxxxxxxxxxxxxxxxxxxxxxxxxxxxxxxxxxxxxxxxxxxxxxxxxxxxxxxxxxxxxxxxxxxxxxxxx

\section{The case of $K = \SU_2$}
\label{sec:SU2}

In this section, we prove Theorem \ref{thrm:SU2Main}. Throughout this section,
we let $K = \SU_2$ so that $G = \SL_2(\C)$.
In this case, every irreducible unitary $K$-module is isomorphic to $(R_d,\SU_2)$
for some $d \geq 1$ where $R_d$ denotes the set of binary forms of degree $d$; similarly, every
irreducible $\SL_2(\C)$-module is isomorphic to some $(R_d,\SL_2(\C))$, see \cite[Section 15.6]{WignerBook}
or \cite[Section 4]{BerndtBook}.
Note that $R_d\simeq S^d\C^2$ with the canonical $\SL_2(\C)$-action.

The proper reductive subgroups of $\SL_2(\C)$ are finite, $\C^\times$, and the normalizer of $\C^\times$,
with the connected component of the identity in the latter equal to $\C^\times$. Hence, if $(V,\SL_2(\C))$
is $2$-large, then for each symplectic slice representation $(W\oplus W^\ast,L)$ such that $L$ is not
principal, we have that $(W,L)$ is $1$-modular by Corollary \ref{cor:ShellSliceProperties}. If $L = \SL_2(\C)$,
then $(W,\SL_2(\C)) = (V,\SL_2(\C))$ and the corresponding $N_S\git L$ is singular by
\cite[Lemma 2.3]{HerbigSchwarzSeaton}. Otherwise, $N_S\git L$ is singular by Lemma \ref{lem:TorusSingular} and
Proposition \ref{prop:SM=PRConnected} so that $(N\git G)_{\pr} = (N\git G)_{\sm}$.

By \cite[Theorem 11.9]{GWSlifting}, we have that the nontrivial finite-dimensional
complex representations $\bigoplus_{d\geq 1} (R_d)^{\oplus m_d}$
of $\SL_2(\C)$ that are not $2$-large are the following:
\[
    R_1^{\oplus k} \;\mbox{for}\; 1 \leq k \leq 3,
    \quad
    R_2,
    \quad
    R_2^{\oplus 2},
    \quad
    R_2\oplus R_1,
    \quad
    R_3,
    \quad
    R_4.
\]
Hence, all other nontrivial $\SL_2(\C)$-modules satisfy the hypotheses of Theorem \ref{thrm:GeneralMain}.
The cases
listed above are exactly those treated in \cite[Theorem 1.2]{Becker}
using a different approach; we will apply Becker's results in some cases but also indicate alternative
methods to treat these cases.

By \cite[Theorem 1.6]{HerbigSchwarzSeaton}, the real symplectic quotients corresponding to the
following representations are each graded regularly symplectomorphic to a linear symplectic orbifold:
\[
    R_1^{\oplus k} \;\mbox{for}\; k = 1,2,
    \quad
    R_2,
    \quad
    R_3,
    \quad
    R_4,
\]
That is, for each $\SL_2(\C)$-module $V$ listed
above, the real regular functions
$\R[M_0]$ on the real symplectic quotient $M_0$ are graded isomorphic to $\R[W\oplus W^\ast]^H$
where $H$ is a finite group and $W$ is a unitary $H$-module (considered as its underlying real vector space).
It is an immediate consequence that the corresponding complex symplectic quotients are graded isomorphic to
$(W_\C\oplus W_\C^\ast)/ H$ where   $W_\C = W\otimes_\R \C$. Such a quotient has
symplectic singularities and is graded Gorenstein by Theorem \ref{thrm:Orbifolds}. Hence,
Theorem \ref{thrm:SU2Main} holds for all nontrivial $\SL_2(\C)$-modules $V$ such that $V^{\SL_2(\C)} = \{0\}$
except for
$R_1^{\oplus 3}$, $R_2^{\oplus 2}$, and $R_2\oplus R_1$. We will consider these three representations individually below,
demonstrating Theorem \ref{thrm:SU2Main} in each case
as well as the stronger fact that the shell $N$ has rational singularities.

Each of the orbifold cases listed above is $1$-large except for $R_1$, $R_1^{\oplus 2}$, and $R_2$.
In the case of $R_2$, $(J)$ is real by \cite[Example 7.14]{ArmsGotayJennings} so that Lemma \ref{lem:R=CSympQuot}
applies. That is, our definition of the complex symplectic quotient coincides with those of Equations
\eqref{eq:defCSympQuotAlgebraic} and \eqref{eq:defCSympQuotComplexGeometric}.
This, however, is not the case for $R_1$ and
$R_1^{\oplus 2}$; see Examples \ref{ex:SL2-R1} and \ref{ex:SL2-2R1} above.

% xxxxxxxxxxxxxxxxxxxxxxxxxxxxxxxxxxxxxxxxxxxxxxxxxxxxxxxxxxxxxxxxxxxxxxxxx

\subsection{$R_2^{\oplus 2}$}
\label{subsec:2R2}

Let $V = R_2^{\oplus 2}$ and recall that $(R_2,\SL_2(\C))$ is isomorphic to the adjoint representation
of $\SL_2(\C)$. Then the action of $\SL_2(\C)$ on $V$ is not effective, as the negative identity acts trivially.
We have $\SL_2(\C)/\{\pm\id\} \simeq\SO_3(\C)$, and $V$
is isomorphic to
$W^{\oplus 2}$ where $W$ is the standard representation of $\SO_3(\C)$ on $\C^3$,
i.e. $V\oplus V^\ast$ is isomorphic to
$W^{\oplus 4}$.

Recall \cite[Theorem 3.4]{HerbigSchwarz} that $V$ is $1$-large so that the ideal $(J)$ generated
by the moment map in
$\R[W^{\oplus 2}]$ is real, hence the ideal $(\mu)$ in
$\C[W^{\oplus 4}]$ is radical.
In \cite[Theorem 5.1]{CapeHerbigSeaton}, it is demonstrated that the variety of the
moment map associated to
$V = W^{\oplus 2}$ has rational singularities. Note that this theorem considers
the corresponding real variety, but the proof implicitly appeals to the complexification
and hence is identical for the complex variety. Hence, by the theorem of Boutot \cite{Boutot},
the complex symplectic quotient $N\git \SO_3(\C) = N\git\SL_2(\C)$
has rational singularities.

Similarly, in \cite[Section 6.2]{CapeHerbigSeaton}, the Hilbert series of the ring $\R[M_0]$
of regular functions on the real symplectic quotient is computed to be
\[
    \Hilb_{\R[M_0]}(t)    =   \frac{1 + 4t^2 + 4t^4 + t^6}{(1 - t^2)^6}.
\]
This of course coincides with the Hilbert series of the regular functions
$\R[M_0]\otimes_{\R}\C$ on the complex symplectic quotient $N\git\SL_2(\C)$.
As
\[
    \Hilb_{\R[M_0]}(1/t) = \frac{t^6(1 + 4t^2 + 4t^4 + t^6)}{(1 - t^2)^6}
        =   (-1)^6 t^6 \Hilb_{\R[M_0]}(t),
\]
Equation \eqref{eq:StanleyGorenHilb} is satisfied with
negative $a$-invariant
equal to $6$, the Krull dimension, so that the complex symplectic quotient is
graded Gorenstein.

We have that $N\git\SL_2(\C)$ is Gorenstein with rational singularities, and by Corollary \ref{cor:SympFormSmooth},
the smooth locus $(N\git\SL_2(\C))_{\sm}$ admits a regular symplectic form. By Theorem
\ref{thrm:NamikawaRatGoren}, $N\git\SL_2(\C)$ has symplectic singularities.
It follows that Theorem \ref{thrm:SU2Main} holds in this case.

% xxxxxxxxxxxxxxxxxxxxxxxxxxxxxxxxxxxxxxxxxxxxxxxxxxxxxxxxxxxxxxxxxxxxxxxxx

\subsection{$R_1^{\oplus 3}$}
\label{subsec:3R1}

Suppose that $V = R_1^{\oplus 3}$ and note that $(V,\SL_2(\C))$ is $1$-large. Then $V\oplus V^\ast$ has
complex dimension $12$ so that as the shell $N$ is a complete intersection by
\cite[Proposition 9.4]{GWSlifting},
$N$ has complex dimension $9$. In addition, $V\oplus V^\ast$ is isomorphic to $R_1^{\oplus 6}$ and has a linear action of
$\GL_6(\C)$ commuting with that of $\SL_2(\C)$. The quotient $(R_1^{\oplus 6})\git\SL_2(\C)$ is isomorphic to the
subspace of $\wedge^2\C^6$ of $2$-forms of rank at most two. There are two $\GL_6(\C)$-orbits
in the quotient, one of them being the origin. Now, a point $v\in R_1^{\oplus 6}$ is in the null cone
$\NN=\NN(R_1^{\oplus 6})$ if and only if the $\SL_2(\C)$-isotropy group of the point is nontrivial, i.e.,
if and only if the six components of $v$ do not span $R_1$. It follows that $\NN$ has complex dimension
$7$ and hence intersects $N$ in complex codimension at least $2$. By Proposition \ref{prop:NormalityShell},
as $(N\smallsetminus\NN)\subset N_{\sm}$, it follows that $N$ is normal. Similarly,
$N\git \SL_2(\C)\subset\wedge^2\C^6$ is smooth except at the orbit of the origin. However,
$\dim_\C N\git \SL_2(\C) = 6$ so that by Theorem \ref{thrm:FlennerCodim} and Corollary \ref{cor:SympFormSmooth},
$N\git \SL_2(\C)$ has symplectic singularities. In fact, the shell $N$ has rational singularities; this
can be verified by a computation similar to, though simpler than, that given in Lemma \ref{lem:R2R1ShellRational}
below.

Similarly, as $(V,\SL_2(\C))$ is $1$-large and $N$ is normal,
it follows from Theorem \ref{thrm:GradeGoren} that $N\git \SL_2(\C)$ is graded Gorenstein.
Alternatively, we have computed the Hilbert series of the regular functions on the (real or complex)
symplectic quotient to be
\[
    \Hilb_{\C[N\git \SL_2(\C)]}(t)
    =
    \frac{1 + 9t^2 + 9t^4 + t^6}{(1 - t^2)^6}.
\]
The $a$-invariant is
$-6 = -\dim_\C\C[N\git \SL_2(\C)]$ implying by Theorem \ref{thrm:StanleyGorenHilb} that
$\C[N\git \SL_2(\C)]$ is graded Gorenstein.

% xxxxxxxxxxxxxxxxxxxxxxxxxxxxxxxxxxxxxxxxxxxxxxxxxxxxxxxxxxxxxxxxxxxxxxxxx

\subsection{$R_2\oplus R_1$}
\label{subsec:R2+R1}

Let $V = R_2 \oplus R_1$, and then $(V, \SL_2(\C))$ is $1$-large so that Lemma \ref{lem:R=CSympQuot}
applies. An explicit description of the generators and relations of the regular functions on
the complex symplectic quotient
$\C[V\oplus V^\ast]^{\SL_2(\C)}/(\mu)^{\SL_2(\C)}$ is computed
in \cite[Proposition 5.3]{Becker}, which is used in \cite[Proposition 5.5]{Becker} to demonstrate
that $N\git\SL_2(\C)$ is a symplectic variety. Moreover, one may use this description to compute the
Hilbert series
\begin{equation}
\label{eq:R2pR1Hilb}
    \Hilb_{\C[N\git \SL_2(\C)]}(t)
    =
    \frac{1 + 2t^2 + 3t^3 + 2t^4 + 2t^5 + 3t^6 + 2t^7 + t^9}
        {(1 - t^2)^2 (1 - t^3) (1 - t^6)}.
\end{equation}
Using this and Theorem \ref{thrm:StanleyGorenHilb}, one computes that the $a$-invariant is $-4$ and
the
complex dimension is $4$ so that $N\git\SL_2(\C)$ is graded Gorenstein.

Let us indicate an alternative method of demonstrating that $N\git\SL_2(\C)$ is a symplectic
variety and computing the Hilbert series. First, we observe that in this case, a variation on the
argument given in \cite[Theorem 5.1]{CapeHerbigSeaton} can be used to prove that the shell $N$
itself has rational singularities.

\begin{lemma}
\label{lem:R2R1ShellRational}
The shell $N$ associated to $(R_2 \oplus R_1,\SL_2(\C))$ has rational singularities.
\end{lemma}
\begin{proof}
Let $V = R_2 \oplus R_1$. We use coordinates for $V\oplus V^\ast$ transforming by weights of $\C^\times\leq\SL_2(\C)$.
For $R_2$ we have $z_2$, $z_0$ and $z_{-2}$ (of weights indicated by the subscripts), and for $R_2^\ast$ we similarly have $z'_2$, $z'_0$ and $z'_{-2}$.
For $R_1$ we have $x$ of weight $1$ and $y$ of weight $-1$, and similarly $x'$ and $y'$ for $R_1^\ast = R_1$.
Then the coordinate functions of the moment map are given by
\begin{align*}
    \mu_1   &=  xx'+z_0z_2'-z_2z_0',
    \\
    \mu_2   &=  xy'+yx'+2(z_{-2}z_2'-z_2z_{-2}'),
    \\
    \mu_3   &=  yy'+z_{-2}z_0'-z_0z_{-2}'.
\end{align*}
As $(V,\SL_2(\C))$ is $1$-large, $(\mu)$ is radical and $N$ is a complete intersection. In particular,
$\C[V\oplus V^\ast]/(\mu)$ is Cohen-Macaulay with Hilbert series $(1 - t^2)^3/(1 - t)^{10}$, hence
the $a$-invariant is $-4$.
A computation of $d\mu$ demonstrates that $N$ is smooth in codimension $1$ and hence
normal. With this, by the theorem of Flenner and Watanabe
\cite[Satz 3.1]{FlennerRational}, \cite[Theorem 2.2]{WatanabeRational},
\cite[Theorem 9.2]{HunekeTight}, it is sufficient to show that $N$ has rational singularities away from the
origin. Because $\SL_2(\C)$ acts on $N$ as isomorphisms, it is sufficient to show that each orbit contains
a point at which $N$ has rational singularities.

At a point where $x$ or $y$ does not vanish, we can by applying an element of $\SL_2(\C)$ assume that both
$x$ and $y$ do not vanish. Localizing at such points, we solve $\mu_1 = 0$ for $x'$ and $\mu_3 = 0$ for $y'$,
and then the equation $\mu_2=0$ can be written
\[
    -2 u v^\prime + 2 v u^\prime = 0
\]
where $u = z_2 y/x - z_0/2$, $v = z_{-2} x/y - z_0/2$,
$u^\prime = z_2^\prime y/x - z_0^\prime/2$, and $v^\prime = z_{-2}^\prime x/y - z_0^\prime/2$.
The corresponding hypersurface is singular only at the origin and hence has rational singularities
by Flenner-Watanabe. The same argument applies near points where $x'$ or $y'$ does not vanish.

Near a nonzero point where $x$, $y$, $x'$, and $y'$ vanish, either some $z_i$ or some $z'_i$ is nonzero.
Assume some $z'_i\neq 0$, and then by moving within an orbit, we may assume that $z'_0 \neq 0$.
Localizing at such points, we solve $\mu_1 = 0$ for $z_2$ and $\mu_3 = 0$ for $z_{-2}$ and then
express $\mu_2 = 0$ as
\[
    x' s + y' t = 0
\]
where $s = y - 2 x z'_{-2}/z'_0$ and $t = x - 2 y z'_2/z'_0$. We again have that this hypersurface
has an isolated singularity at the origin and negative $a$-invariant and hence rational singularities
by Flenner-Watanabe. The same argument again applies near points where some $z_i$ is nonzero, completing
the proof.
\end{proof}

Finally, we describe a direct computation of the Hilbert series of $N\git G$ given in Equation
\eqref{eq:R2pR1Hilb}. Consider $\mathfrak{g}$ as the quadratic equations defining the shell $N$. We again have
that the elements of $\mathfrak{g}$ are a regular sequence. Let $S$ denote
$\C[V\oplus V^*]$, and then we have an exact sequence
\[
    0   \to S\simeq S\otimes\wedge^3\mathfrak{g}
        \to S\otimes \wedge^2\mathfrak{g}
        \to S\otimes\mathfrak{g}\to S
        \to S/\mathfrak{g}\to 0.
\]
Taking $G = \SL_2(\C)$-invariants yields
\[
    0   \to S^G
        \to (S\otimes \wedge^2\mathfrak{g})^G
        \to (S\otimes\mathfrak{g})^G
        \to S^G\to (S/\mathfrak{g})^G
        \to 0.
\]
Now, the elements of $\mathfrak{g}$ are in degree $2$, the elements of $\wedge^2\mathfrak{g}$ are in degree $4$,
and $\wedge^3\mathfrak{g}$ is in degree $6$. Let $a(t)$ denote the Hilbert series of $S^G$ and let $b(t)$ denote
the Hilbert series for the occurrences of $\mathfrak{g}$ in $S$. Then the Hilbert series  of $(S/\mathfrak{g})^G$
is
\[
    \Hilb_{\C[N\git \SL_2(\C)]}(t)    =   a(t)-t^2 b(t)+t^4b(t)-t^6a(t).
\]
Now one can use the program Lie \cite{Lie} to calculate $a(t)$ and $b(t)$ to any degree.
Going up to degree $14$, we found that
$\Hilb_{\C[N\git \SL_2(\C)]}(t)$ is the rational
function given by Equation \eqref{eq:R2pR1Hilb}. Using classical invariant theory, one
can compute that $\C[N\git \SL_2(\C)]$ has a homogeneous regular sequence consisting
of two elements of degree $2$, one of degree $3$, and one of degree $6$ such that the
numerator of the Hilbert series has degree at most $14$. Then the Lie calculation
confirms Equation \eqref{eq:R2pR1Hilb}.

% xxxxxxxxxxxxxxxxxxxxxxxxxxxxxxxxxxxxxxxxxxxxxxxxxxxxxxxxxxxxxxxxxxxxxxxxx
% xxxxxxxxxxxxxxxxxxxxxxxxxxxxxxxxxxxxxxxxxxxxxxxxxxxxxxxxxxxxxxxxxxxxxxxxx
% xxxxxxxxxxxxxxxxxxxxxxxxxxxxxxxxxxxxxxxxxxxxxxxxxxxxxxxxxxxxxxxxxxxxxxxxx

\bibliographystyle{amsalpha}

\def\cprime{$'$}
\providecommand{\bysame}{\leavevmode\hbox to3em{\hrulefill}\thinspace}
\providecommand{\MR}{\relax\ifhmode\unskip\space\fi MR }
\providecommand{\MRhref}[2]{%
  \href{http://www.ams.org/mathscinet-getitem?mr=#1}{#2}
}
\providecommand{\href}[2]{#2}

\end{document}